\documentclass[10pt]{article}
\usepackage{amsmath,amsfonts,amssymb,amsthm,epsfig,color}
\usepackage{amsmath}
\usepackage{amssymb}
\usepackage{amsfonts}
\usepackage{mathrsfs}
\usepackage{amsthm}
\usepackage{epsfig,epic}
\usepackage[english]{babel}
\usepackage{graphics}
\usepackage{graphicx}
\usepackage{eucal}
\usepackage{amscd}
\usepackage{amssymb}
\usepackage{fancyhdr}
\usepackage[hang]{subfigure}
\usepackage{array}
\usepackage[nottoc]{tocbibind}
\usepackage{verbatim} 

\newtheorem*{thm*}{Theorem}
\newtheorem{thm}{Theorem}[section]
\newtheorem{prop}{Proposition}[section]
\newtheorem{lemma}{Lemma}[section]
\newtheorem{cor}{Corollary}[section]

\newtheorem{defn}{Definition}[section]
\newtheorem{rem}{Remark}[section]

\newcommand{\RL}[1]{Z_{#1}}
\newcommand{\E}{\mathscr{E}}

\newcommand{\RLs}{A}

\newcommand{\R}{\mathcal{R}}
\newcommand{\V}{\mathcal{V}}
\newcommand{\Z}{\mathcal{Z}}
\newcommand{\RLp}[1]{Z^{(#1)}}
\newcommand{\kei}[1]{{k(I)}}

\newcommand{\Orb}[2]{\mathcal{O}_{#2}(#1)}
\newcommand{\Or}{\mathcal{O}}

\newcommand{\Th}[2]{\Theta^{#1}_{#2}}
\newcommand{\norm}[1]{|\!| #1 |\!|}
\newcommand{\fracpart}[1]{\{\!\{ #1 \}\!\}}

\newcommand{\be}{\begin{equation}}
\newcommand{\ee}{\end{equation}}
\newcommand{\bes}{\begin{equation*}}
\newcommand{\ees}{\end{equation*}}
\newcommand{\st}{\,  | \, \, }

\newcommand{\Sym}[1]{\mathcal{S} _{ #1} }
\newcommand{\ud}{\mathrm{d}}
\newcommand{\Leb}[1]{Leb\left( #1\right)}

\newcommand{\BS}[2]{S_{#2}(#1)}

\newenvironment{proofof}[2]{\begin{proof}[Proof of #1 \ref{#2}.]}{\end{proof}}

\title{Absence of mixing in area-preserving flows on surfaces}

\author{Corinna Ulcigrai}

\date{December 6, 2008}
\begin{document}

\maketitle
\begin{abstract}
We prove that  minimal area-preserving flows locally given by a smooth Hamiltonian on a closed surface of any genus  are  typically (in the measure-theoretical sense) not mixing. The result is obtained by considering special flows over interval exchange transformations under roof functions with symmetric logarithmic singularities  and proving absence of mixing for a full measure set of interval exchange transformations.
% that such flows are typically \emph{not mixing}. As a corollary, minimal flows given by multi-valued Hamiltonians on higher genus surfaces which are minimal and have only simple non-degenerate saddles are typically not mixing.
\end{abstract}

%\include{AbsenceIntro}
%\include{Absencecocycle}
%\include{AbsenceRauzy}

%\subsection{Statements of the main Theorems}\label{statementssec}

%\section{Introduction}
\section{Definitions and Main Results}
\subsection{Flows given by multi-valued Hamiltonians}\label{multivalHsec}
Let us consider the following natural construction of area-preserving flows on surfaces.  %which are given by a multi-valued Hamiltonian as follows.   
%The following class of area-preserving flows on surfaces can be described using suspension flows over IETs with logarithmic singularities and shows why such flows arise naturally.
On a closed, compact, orientable surface of genus $g\geq 1$ with a fixed smooth area form, consider a smooth closed differential $1$-form $\omega$.
%The form $\omega$ determines a flow  $\{\varphi_t\}_{t\in\mathbb{R}}$ %given by a \emph{multi-valued Hamiltonian} as follows. 
 % a closed surface of genus $g\geq 2$ with a fixed area form and a closed differential $1$-form $\omega$ on it. 
Since $\omega $ is closed, it is locally given by $\ud H$ for some real-valued function $H$.  The flow $\{\varphi_t\}_{t\in\mathbb{R}}$ determined by $\omega$ is the associated Hamiltonian flow, which is given by local solutions of $\dot{x}=\frac{\partial H}{\partial y}$, $\dot{y}=-\frac{\partial H}{\partial x}$. The flow $\{\varphi_t\}_{t\in\mathbb{R}}$ is known as the \emph{multi-valued Hamiltonian flow}\footnote{The Hamiltonian $ H$ locally defined by $\ud H= \omega$  could indeed not be  well defined globally (see \cite{NZ:flo}, \S1.3.4) and hence determines a multi-valued Hamiltonian function.} determined by $\omega$. 
Remark that the transformations $\varphi_t$, for each $t \in \mathbb{R}$, are symplectic, which in dimension $2$ is equivalent to area-preserving.

The study of flows given by multi-valued Hamiltonians was initiated
by S.P. Novikov \cite{No:the} in connection with problems arising in solid-state
physics i.e., the motion of an electron in a metal under the action of a magnetic field.
The orbits of such flows arise also in pseudo-periodic topology, as hyperplane
sections of periodic surfaces in $\mathbb{T}^n$ (see e.g. Zorich \cite{Zo:how}).%}

From the point of view of topological dynamics, a decomposition into  \emph{minimal components} (i.e.~subsurfaces on which the flow is minimal) and components on which all orbits are periodic was  proved independently  by Maier \cite{Ma:tra}, by Levitt \cite{Le:feu} (in the context of foliations on surfaces) and by Zorich \cite{Zo:how} for multi-valued Hamiltonian flows. We consider the case in which the  flow is \emph{minimal}, i.e.~all semi-infinite trajectories are dense.

From the point of view of ergodic theory,  one is naturally lead to ask whether the flow on each minimal component is \emph{ergodic} and, in this case, whether it is \emph{mixing}.  
Ergodicity is equivalent to ergodicity of the Poincar{\'e} first return map on a cross section, which is isomorphic to a minimal interval exchange transformation (see \S \ref{IETsdefsec} for definitions). A well-know and celebrated result asserts that \emph{typical}\footnote{The notion of typical here is measure-theoretical, i.e.~it refers to \emph{almost every} IET in the sense defined before the statement of Theorem \ref{absencethm}.} IETs are uniquely ergodic (\cite{Ve:gau, Ma:int}). 
%Mixing, on the other side, is a property which is sensitive to reparametrizations of the flow. 

In this paper we address the question of mixing.  Let $\mu$ be the area renormalized so that $\mu(S)=1$. Let us recall that $\{ \varphi_t\}_{t\in \mathbb{R}}$ %on the probability space $S$ preserving the measure $\mu$ 
is said to be \emph{mixing} if for each pair $A$, $B$ of Borel-measurable sets one has
\be \label{mixingdef}
\lim_{t\rightarrow \infty} \mu(\varphi_t(A)\cap B)=\mu(A)\mu(B). 
\ee 
We recall that a \emph{saddle connection} is a flow trajectory which contains both an incoming and an outgoing saddle separatrix and that if a flow has no saddle connections then it is minimal \cite{Ma:tra}. 
The main result is the following.
\begin{thm}\label{multivalHthm}
Let $S$ be closed surface of genus $g\geq 2$ with a fixed area form and let $\{\varphi_t\}_{t\in\mathbb{R}}$ be the flow given by a multi-valued Hamiltonian associated to  a closed differential $1$-form $\omega$. Assume that  $\{\varphi_t\}_{t\in\mathbb{R}}$ has only non-degenerate saddles and no saddle connections.   For a \emph{typical} such %(in the measure-theoretical sense)
 closed form $\omega$ the flow $\{\varphi_t\}_{t\in\mathbb{R}}$  is \emph{not  mixing}. 
\end{thm}
\noindent The notion of \emph{typical} here is measure-theoretical and refers to the natural measure-class on such flows, which is obtained by pulling-back the Lebesgue measure class by the period-map. We explain the precise meaning of typical in \S\ref{reductionmultivalHsec}.  

Theorem \ref{multivalHthm} settles the open question (which appears for example in Forni \cite{Fo:dev} and in the survey \cite{KT:spe} by Katok and Thouvenot, \S6.3.2.) of whether a typical minimal multi-valued Hamiltonian flow with only simple saddles is mixing. 
%Let us remind the definitions of suspension flows and interval exchange transformations and formulate 
Even if non-mixing, such flows are nevertheless typically \emph{weakly mixing}\footnote{\label{wmdef}Let us recall that a flow $\{\varphi_t\}_{t \in \mathbb{R}}$ preserving a probability measure $\mu$ is \emph{weakly mixing} if for each pair $A$, $B$ of measurable sets $\frac{1}{T}\int_{0}^T \left|  \mu(\varphi_t(A)\cap B) - \mu(A)\mu(B)\right| \ud t $  converges to zero as $T$ tends to infinity.} (\cite{Ul:wm}, see also \S \ref{historysec}).

Let us remark that all assumptions of Theorem \ref{multivalHthm} are crucial for the absence of mixing. Indeed, if a minimal flow has  multi-saddles, corresponding to higher order zeros of $\omega$, then $\varphi_t$ is mixing, as proved by Kochergin \cite{Ko:mix}. 
On the other hand, if the surface flow has not only simple saddles but also centers or more than one minimal component (more precisely, if there are saddle loops homologous to zero) than one can produce mixing (\cite{SK:mix, Ul:mix}, see also \S \ref{historysec}). 
  The asymptotic behavior of Birkhoff sums and its deviation spectrum for this class of flows was described by Forni in \cite{Fo:dev}.  

In the next section we recall the definitions of interval exchange transformations and special flows and  formulate the main Theorem in the setting of special flows (Theorem \ref{absencethm}, from which  Theorem  \ref{multivalHthm} will be deduced). 
Previous results on ergodic properties of special flows over IETs are recalled in \S \ref{historysec}.

\subsection{Special flows with logarithmic singularities}
Special flows give a useful tool to describe area-preserving flows on surfaces. When representing a flow on a surface (or one of its minimal components) as a special flow, it is enough to consider a transversal to the flow: the first return, or Poincar{\'e} map, to the transversal determines the base transformation $T$, while the function $f$ gives the first return time of the flow to the transversal. Different functions $f$ describe different time-reparametrizations of the same flow, hence give rise to flows which, topologically, have the same orbits. Interval exchange transformations arise naturally as first return maps (up to smooth reparametrization), see \S\ref{reductionmultivalHsec}. 

%The main Theorem that we prove is the following. 
%\begin{thm}\label{absencethm} For every irreducible $\pi$, for Lebesgue-a.e. lengths vector $\underline{\lambda}\in %\Delta_{d-1}$, 
%the suspension flow $\{\varphi_t \}_{t\in \mathbb{R}}$ built over $T=(\underline{\lambda}, \pi)$ under a roof function $f$ with  symmetric logarithmic singularities is \emph{not mi\Thetang}. \end{thm}
%Hence this type of suspesion flows over  a \emph{typical} IET are weakly mixing.

%Since as recalled in \S \ref{multivalHsec} flows on surfaces given by multi-valued Hamiltonians can be represented as flows over IETs with this type of singularities, we get as a Corollary the following.

\paragraph{Interval exchange transformations.}\label{IETsdefsec}
Let $I^{(0)}=[0,1)$,  let $\pi\in \Sym{d}$, $d \geq 2$, be a permutation\footnote{We are using here the notation for IETs classically introduced by Keane \cite{Ke:int} and Veech \cite{Ve:int, Ve:gau}. We remark that recently by Marmi-Moussa and Yoccoz introduced a new labeling of IETs (see the lecture notes by Yoccoz \cite{Yo:con} or Viana \cite{Vi:IET}), which considerably facilitates the analysis of Rauzy-Veech induction. We do not recall it here, since it does not bring any simplification to our proofs.} and  let $\Delta_{d-1}$ denote the simplex of vectors $\underline{\lambda}\in \mathbb{R}_+^d$ such that $\sum_{i=1}^{d}\lambda_i=1$.   
The \emph{interval exchange transformation} (IET) of $d$ subintervals given by  $(\underline{\lambda}, \pi)$ with $\underline{\lambda} \in \Delta_{d-1}$  is the map $T:I^{(0)}\rightarrow I^{(0)}$ given by\footnote{The sums in the definition are by convention zero if over the empty set, e.g.~for $j=0$.} 
\bes
 T(x) =  x  - \sum_{i=1}^{j-1} \lambda_i + \sum_{i=1}^{j-1} \lambda_{\pi^{-1}(i)} \quad \mathrm{for} \quad x \in I^{(0)}_{j} = [ \sum_{i=1}^{j-1} \lambda_i , \sum_{i=1}^{j} \lambda_i ), \qquad j=1, \dots , d.
\ees
\noindent In other words  $T$ is a piecewise isometry which rearranges the subintervals of lengths given by $\underline{\lambda}$ in  the order determined by $\pi$. We shall often use the notation $T=(\underline{\lambda}, \pi)$.
Let $\Sigma_{\underline{\lambda},\pi} = \{\sum_{i=1}^{j} \lambda_i, \quad j=1,\dots, d  \} \cup \{0 \} $ be the set of \emph{discontinuities } of   $T$ together with the endpoints of $I^{(0)}$. 
We say that $T$ is \emph{minimal} if the orbit of all points  are dense. 
We say that the permutation $\pi \in  \Sym{d}$ is \emph{irreducible} if, whenever the subset $\{ 1, 2, \dots k \}$ is $\pi$-invariant, then $k=d$. Irreducibility is a necessary condition for minimality. Recall that $T$ satisfies the \emph{Keane condition} if the orbits of all discontinuities in $ \Sigma_{\underline{\lambda},\pi} \backslash \{ 0, 1 \} $ are infinite and disjoint. If $T$ satisfies this condition, then $T$ is minimal \cite{Ke:int}.
% so that under the action the transformation $I^{(0)}_i$ becomes the $\pi(i)$$^{th}$ interval, i.e. the order of the subintervals after applying $T$ is $I^{(0)}_{\pi^{-1}(1)}$,$I^{(0)}_{\pi^{-1}(2)}$ $\dots $ $I^{(0)}_{\pi^{-1}(d)}$ (see Figure \ref{IET}).
%More precisely, let $\Delta_{d-1}$ denote the simplex of vectors $\underline{\lambda}\in \mathbb{R}_+^d$ such that $\sum_{i=1}^{d}\lambda_i=1$. 
%\begin{defn} The IET $(\lambda, \pi)=T$ associated to   $\pi\in \Sym{d}$, $\underline{\lambda}\in \Delta_{d-1}$, is 
%Let $\Sigma= \Sigma_{\underline{\lambda}, \pi}$ be  the finite set of discontinuities of $T=(\lambda, \pi)$, i.e.  $\Sigma_{\underline{\lambda}, \pi} = \{ 0 \} \cup \{ \sum_{i=1}^{j-1} \lambda_i , i=2, \dots, d \}$. 
%\begin{defn}\label{irreducibledef}
%\end{defn}
%One can see that, given  an exchange $T$ of $d$ intervals, the  iterates $T^n= T\cdot \dots \cdot T$ ($n$ times) are again IETs of at most $n(d+2)$ subintervals. Also, given a subinterval $J\subset I^{(0)}$, one can consider the map $T'$  induced on $J$ by $T$, i.e. the map $T'(x)= T^{n(x)}x$ where $n(x)$ is the minimum $n\in \mathbb{N}$ such that $T^n x\in J$, which is well defined for a.e. $x$ by Poincar{\'e} recurrence. The induced map $T'$ is again an IET of at most $d+2$ subintervals. We refer e.g. to \cite{CFS:erg}. Some of the ergodic properties of IETs are recalled in \S\ref{suspIETsec}.

\paragraph{Special flows.}\label{suspflowsdefsec}
Let $f\in L^1 (I^{(0)}, dx)$ be a strictly positive function %$f\geq m_f >0$
 and assume $\int_{I^{(0)}} f(x) dx =1$.  
%Further assumptions on $f$ will be formulated in \S \ref{symmetriclog}.
Let $X_f \doteqdot \{ (x,y) \in \mathbb{R}^2 \st  x \in I^{(0)}, \, 0\leq y < f(x) \}$ be the set of points below the graph of the roof function $f$ and $\mu$ be the restriction to $X_f$ of the Lebesgue measure $\ud x\,\ud y$. %, which is the \emph{phase space $X_f$} of the suspension flow defined as follows.  Introduce the normalized measure 
Given  $x\in I^{(0)}$ and $r\in \mathbb{N}^+$ we denote by  $\BS{f}{r}( x)  \doteqdot \sum_{i=0}^{r-1} f(T^i(x))$ 
the $r^{th}$ non-renormalized \emph{Birkhoff sum} of $f$ along the trajectory of $x$ under $T$. By convention, $\BS{f}{0}( x) \doteqdot 0$.  
Let $t>0$. Given $x\in I^{(0)}$ denote by $r(x,t)$ the integer uniquely defined by
$r(x,t)\doteqdot \max \{ r\in \mathbb{N} \, | \quad  \BS{f}{r}(x) < t \}$.

The \emph{special flow built over}\footnote{One can define in the same way special flows over any measure preserving transformation $T$ of a probability space $(M, \mathscr{M}, \mu)$, see e.g. \cite{CFS:erg}.}  $T$ \emph{under the roof function f} is a one-parameter group $\{ \varphi_t \}_{t\in \mathbb{R}}$ of $\mu$-measure preserving transformations of $X_f$  
\begin{comment}
\begin{figure}
\centering
\includegraphics[width=0.53\textwidth]{symlog.eps}
\caption{\label{suspensionflow} The motion of the point $(x,0)$ under $\varphi_s$ for $0\leq s \leq t$.}
\end{figure}
\end{comment}
whose action is given, for $t>0$, by 
\be \label{flowdef}
\varphi_t(x,0) = \left( T^{r(x,t)}(x), t- \BS{f}{r(x,t)}(x)\right).
\ee
For $t<0$, the action of the flow is defined as the inverse map and $\varphi_0$ is the identity.  
Under the action of the flow a point of $(x,y) \in X_f$  moves with unit velocity along the vertical line up to the point $(x,f(x))$, then jumps instantly to the point $\left( T(x),0 \right)$, according to the base transformation.    
%(see Figure \ref{suspensionflow}). 
Afterward it continues its motion along the vertical line until the next jump and so on. 
 The integer $r(x,t)$ gives the number of \emph{discrete iterations} of the base transformation $T$ which the point $(x,0)$ undergoes when flowing up to time $t>0$. %

\paragraph{Logarithmic singularities.} \label{symmetriclog}
We consider the following class of roof functions with logarithmic symmetric singularities.  The motivation for considering special flows over IETs under such roofs is explained in  \S\ref{reductionmultivalHsec}.
%\paragraph{Symmetric logarithmic singularities.} \label{asymmetriclog}
%We consider the following class of roof functions with a logarithmic symmetric singularity at the endpoints of the interval. 
Let us introduce two auxiliary functions $u$, $v$ defined on $(0,1)$ as follows %and extended to be 1-periodic on the real line:
\bes
u(x) := \frac{1}{ x}, \qquad v(x) := \frac{1}{1- x}, 
%\quad for x \in (0,1); \qquad u(x)=u(\{x\}), v(x)=v(\{x\})  for x \in \mathbb{R};
\ees
% where $\{x\}$ denotes the fractional part of $x$. 
and extended to the whole real line so that they are periodic of period 1, i.e. for $x\in \mathbb{R}$, $u(x)=u(\fracpart{x})$ and  $v(x)=v(\fracpart{x})$ where $\fracpart{x}$ denotes the fractional part of $x$. 
Let $0\leq \overline{z}_0^+< \overline{z}_1^+ <\dots <\overline{z}^+_{s_1-1} < 1$ be the $s_1$ points where the roof function is right-singular (i.e.~the right limit is infinite) and 
$0<\overline{z}_0^-< \overline{z}_1^- <\dots <\overline{z}^-_{s_2-1} \leq 1$ the $s_2$  points where the roof function is left-singular (i.e.~the left limit is infinite). 
 Let us denote  $u_i (x) = u(x-\overline{z}^+_i)$ for  $i=0,\dots , s_1-1$ and $v_i(x)= v(x-\overline{z}^-_i)$, $i=0,\dots , s_2-1$. 
\begin{defn}%[Symmetric logarithmic singularity] 
\label{symmsingdef} The function $f$   has \emph{logarithmic singularities} at $0\leq \overline{z}_0^+<  \dots <\overline{z}^+_{s_1 -1} < 1$ (right singularities) and $0<\overline{z}_0^- <\dots <\overline{z}^-_{s_2-1} \leq 1$ (left singularities) if 
%$0\leq \overline{z}_0 < \overline{z}_1 <\dots <\overline{z}_{s-1} < \overline{z}_{s} \leq 1$ if 
%\begin{itemize
%\item[(i)]
$f \in \mathscr{C}^2$ on the complement of the union of the singularities 
%\left(I^{(0)}\backslash \{\overline{z}^+_0,\dots , \overline{z}^+_{s_1-1},  \overline{z}^-_0,\dots , \overline{z}^-_{s_2-1} \}\right)$ 
and 
%;\item[(ii)] 
there exists constants $C_i^+$  $i=0,\dots, s_1-1$  and $C_i^-$ for  $i=0,\dots, s_2 -1$ and a function $g$ of bounded variation on $[0,1]$, such that
\be\label{logsing}
f'=   \sum_{i=0}^{s_2 -1}  C_i^-  v_i  -  \sum_{i=0}^{s_1 -1} C_i^+ u_i  + g.
\ee
%\end{itemize}
The logarithmic singularities are called \emph{symmetric} if moreover $\sum_{i=0}^{s_2 -1} C_i^- = \sum_{i=0}^{s_1-1} C_i^+$.
%\bes
% \sum_{i=0}^{s-1} C_i^- = \sum_{i=0}^{s-1} C_i^+ %\qquad (symmetry)
%\ees
\end{defn}
\noindent %%%%%%ELIMINABILE
We remark that the derivative $f'$ of a function with symmetric logarithmic singularities is not integrable. 
An example of a function with logarithmic singularities at  $\overline{z}_0^{+}=0$, $\overline{z}_0^{-}=1$ and $\overline{z}_i^{\pm} := \overline{z}_i$ for $i=1,\dots, s-1$ (here $s_1=s_2=s$) is given by 
%The model example is given by the function
%\be\label{logsymex}
%$f(x) = \frac{1}{2}\left(|\ln x | + | \ln (1-x)| \right)$,
%\ee
%for which $f'(x)=\frac{1}{2}\left( \frac{1}{1-x}-\frac{1}{x}\right)$ and $C_0^+=C_1^-=1/2$. More generally, an example of a function which has derivative as in (\ref{logsing}) is
\bes\label{logsymex}
f(x) = C_0^+ |\ln(x)| + \sum_{i=1}^{s-1} \left( C_i^+ |\ln \left(\left\{x-\overline{z}_i\right\}\right)| + C_i^-  |\ln \left(\left\{\overline{z}_i - x\right\}\right)| \right) +  C_0^- |\ln(1-x)|.
\ees
%Its graph has the form shown in Figure \ref{suspensionflow}.
%Definition \ref{symmsingdef} implies that
%\be\label{ln}
%\lim_{x\rightarrow 0^+} \frac{f(x)}{|\ln x|}=C; \qquad \quad  \lim_{x\rightarrow 1^-} \frac{f(x)}{|\ln (1-x) |}=C .
%\ee
%When the constants which appear in the two limits in (\ref{ln}) are different, than the singularity is called \emph{asymmetric}.  The case of asymmetric singularities was considered in \cite{Ul:mix}.
%For convenience, in what follows we will always assume that $C=1$.

\paragraph{Absence of mixing}\label{statementssec}
%As recalled in \S\ref{susprotsec}, suspension flows over a rotation under a roof with a symmetric logarithmic singularity are weak mixing, but never mixing for any $\alpha\in\mathbb{Q}$. % and  
%%%In the case of a suspension flow %with symmetric logarithmic singularities 
%over an interval exchange, the question of the presence/absence of mixing when $f$ has symmetric logarithmic singularities is much more subtle, since, as recalled in \S\ref{suspflowsIETssec}, IETs exhibit stronger ergodic properties than rotations (e.g. they are generically weakly mixing \cite{AF:wea} and display power-deviations from ergodic averages \cite{Zo:dev}).
%Let us recall that, in the case of the
%In the case of a suspension flow %with symmetric logarithmic singularities 
%over an interval exchange, the question of the presence/absence of mixing when $f$ has symmetric logarithmic singularities is much more subtle, since IETs exhibit stronger ergodic properties than rotations (for instance they are generically weakly mixing \cite{AF:wea} and display power-deviations from ergodic averages \cite{Zo:dev}). 
%A typical IET ($d\geq 4$)exhibit stronger ergodic properties than a rotation, since, as recalled above, they are generically weakly mixing,  and moreover display power-deviations from ergodic averages \cite{Zo:dev} (while for rotations deviations for ergodic averages are bounded). Hence it is reasonable to expect stronger
The main Theorem that we prove in this context is the following. Here and in the rest of the paper we will say that a result holds for \emph{almost every IET} if it holds for any \emph{irreducible} permutation $\pi$ on $d\geq 2$ symbols and a.e.~choice of the length vector $\underline{\lambda}\in \Delta_{d-1}$ with respect to the restriction of the $d$-Lebesgue measure to the simplex $\Delta_{d-1}$.  
%Recall that a flow $\{ \varphi_t\}_{t\in \mathbb{R}}$ on the probability space $X$ preserving the measure $\mu$ is said to be \emph{mixing} if for each pair of measurable sets $A$, $B$, one has
%\be \label{mixingdef}
%\lim_{t\rightarrow \infty} \mu(\varphi_t(A)\cap B)=\mu(A)\mu(B). 
%\ee
\begin{thm}\label{absencethm} For almost every IET $T$ 
%irreducible $\pi$, for Lebesgue-a.e. lengths vector $\underline{\lambda}\in \Delta_{d-1}$, 
the special flow $\{\varphi_t \}_{t\in \mathbb{R}}$ built over $T =(\underline{\lambda}, \pi)$ 
under a roof function $f$ with  symmetric logarithmic singularities at a subset of the singularities $\Sigma_{\underline{\lambda},\pi}$ of $T$ is \emph{not mixing}. \end{thm}
%Hence this type of suspesion flows over  a \emph{typical} IET are weakly mixing.
\noindent 
It is worth remarking that nevertheless special flows with logarithmic singularities over  typical IETs are weakly mixing\footnote{
The definition of weak mixing was recalled in Footnote \ref{wmdef} (page \pageref{wmdef}).
%Let us recall that the flow $\{\varphi_t\}_{t \in \mathbb{R}}$ is weakly mixing if the only solutions to the equation $g(\varphi_t x) = e^{i t s } g(x)$ with $g$ measurable and $s \in \mathbb{R}$ are the constant functions for $s=0$.
}, as proved by the author in \cite{Ul:erg, Ul:wm}.  We show in \S \ref{reductionmultivalHsec} that flows on surfaces given by multi-valued Hamiltonians can be represented as flows over IETs with logarithmic  singularities and that Theorem \ref{multivalHthm} can be deduced from Theorem \ref{absencethm}.
 %Hence, flows which arise from multi-valued Hamiltonians provides a class of examples of dynamical systems wich are   

%, we get as a Corollary the following.
%%%\subsubsection{Outline of the paper}
\subsection{Ergodic properties of logarithmic special flows}\label{historysec}
\paragraph{Flows over rotations}
%\subsection{Ergodic properties of flows with logarithmic symmetric singularities}\label{historysec}
Assume first that
%In the case when
 the base transformation is a rotation of the circle (i.e. the map $R_{\alpha} x = x+ \alpha \mod 1$), which can be seen as an interval exchange of $d=2$ intervals. Kochergin proved in \cite{Ko:non} that   special flows with \emph{symmetric} logarithmic singularities  \cite{Ko:non} are not mixing for a.e.~$\alpha$. Recently, in \cite{Ko:non2}, he shows that absence of mixing holds for all $\alpha$. An intermediate result for $s_1=s_2=1$ and $C_0^-=C_0^+$ is a consequence of \cite{Le:sur}. 
 %The investigation of flows with logarithmic singularities over rotations was pushed further by a series of works by Fr{\c a}czek and Lema{\'{n}}czyk. 
In \cite{FL:ons} Fr{\c a}czek and Lema{\'{n}}czyk consider the %logarithmic
 roof function $f(x) = |\ln x |/2 $$+ $$ | \ln (1-x)|/2$ %(\ref{logsymex})
 and show that the corresponding special flow over $R_{\alpha}$ is \emph{weakly mixing} for all $\alpha$.  They also
 %%% ACCORCIABILE
  push the investigation to more subtle spectral properties, showing  in \cite{FL:acl} that such flows are spectrally disjoint from all mixing flows. 

On the other hand, if the roof has \emph{asymmetric logarithmic singularities}, % i.e. in Definition \ref{logsingdef} $\sum C_i^+ \neq \sum C_i^-$. we have 
 Khanin and Sinai proved  in \cite{SK:mix} that, under a diophantine condition on the rotation angle which holds for a full measure set of $\alpha$, the corresponding special flow  is \emph{mixing}, answering affirmatively to a question asked by Arnold in  \cite{Ar:top}. The diophantine condition of \cite{SK:mix} was weakened by Kochergin %in \cite{Ko:nonI} 
in a series of works (\cite{Ko:nonI, Ko:nonII, Ko:som, Ko:wel}).

\paragraph{Flows over IETs}
In \cite{Ul:erg, Ul:mix} the author proved that special flows over typical IETs under a roof function $f$ having a single \emph{asymmetric} logarithmic singularity at the origin (i.e.~as in Definition \ref{symmsingdef} with $s_1=s_2=1$ and $C_0^+\neq C_0^-$) are \emph{mixing}. The same  techniques can be applied  to the situation of several logarithmic singularities as long as the roof satisfy the asymmetry condition  $\sum C_i^+ \neq \sum C_i^-$. %The proof uses diophantine conditions expressed thouhg the Rauzy Veech algorithm. 
Let us also recall that if the singularities instead than logarithmic are power-like (i.e. $f$ blows up near singularities as $1/x^{\alpha}$ for $\alpha >0$), then  
mixing was proved by Kochergin, \cite{Ko:mix}.

If the singularities are \emph{symmetric}, two results in special cases were recently proved. %one could expect that the corresponding flow over IETs is not mixing, as in the case of  rotations;  are the following. 
The author showed in \cite{Ul:erg} the absence of mixing if the IET on the base satisfies a condition which is similar to $\alpha$ being bounded type for rotations (which in particular holds only for a mesure zero set of IETs).  Scheglov recently showed in \cite{Sch:abs} that if $\pi=(54321)$, for a.e.~$\underline{\lambda}$ the special flow over $(\underline{\lambda}, \pi)$ under a particular class of functions  with symmetric logarithmic singularities\footnote{A function $f$ with symmetric logarithmic singularities, as defined in  \cite{Sch:abs}, is such that $f'$ is a linear combination of the functions $f_i(x)=1/(b_i-x) -1/(x-a_i)$ defined on the interior of the IET subintervals $I^{(0)}_i = [a_i,b_i)$ for $i=1,\dots ,d$ and the function $1/(1-x)-1/x$. In particular $s_1=s_2$ and constants come in pairs $\{C_i^+,C_i^-\}$ such that $C_i^+=C_i^-$. Thus, this class is more restrictive than the one given by Definition \ref{symmsingdef}. \label{Scheglovsymm}%This amounts for the assumption that the surface flow has two \emph{isometric} saddles.
}  is not mixing, from which it follows that Theorem \ref{multivalHthm} holds in the special case in which $g=2$ and the flow has two isometric saddles. Unfortunately, his methods does not extend to higher genus, for the reasons explained in the remark at the end of \S \ref{cancelletionssec} (page \pageref{Scheglov}).  

%On the other side, let us recall that the author also proved that suspension flows with logarithmic singularities over a typical IET are nevertheless weakly mixing\footnote{Let us recall that the flow $\{\varphi_t\}_{t \in \mathbb{R}}$ is weakly mixing if the only solutions to the equation $g(\varphi_t x) = e^{i t s } g(x)$ with $g$ measurable and $s \in \mathbb{R}$ are the constant functions for $s=0$.} \cite{Ul:erg, Ul:wm}. 

It is worth recalling also that interval exchanges themselves are not mixing and that special flows over IETs are never mixing if the function $f$ is of bounded variation (both results were proved by Katok in \cite{Ka:int}).  
On the other side, Avila and Forni \cite{AF:wea} showed that IETs which are not of rotation-type are typically weakly mixing and that special flows over IETs with piecewise constant roofs are also typically weakly mixing. 

\section{Background on cocycles and Rauzy-Veech induction}
\subsection{Some properties of cocycles}
Let $(X,\mu, F)$ be a discrete dynamical system, where $(X, \mu)$ is a probability space and $ F$ is a $\mu$-measure preserving map on $X$. A a measurable map $A : X  \rightarrow SL(d,\mathbb{Z})$ ($d\times d$ invertible matrices) determines a cocycle $A$ on  $(X,\mu, F)$. If we  denote by  $ A_n (x) = A(F^n x)$ and by 
 $A_F^{n}(x) = A_{n-1} (x)  \cdots A_1(x)  A_0 (x) $,  
%\bes
%A^{(n)}(x) = A(F^{n-1}x) \cdots A(F x) A(x)
%\ees
%$A_n (x) = A(F^n x)$ and by $A^{(n)}$ the product 
%$ A_n A_{n-1} \dots A_1$, 
the following \emph{cocycle identity}
\be \label{cocycleid}
A_F^{m+n}(x) = A_F^{m}(F^n x)A_F^{n}(x)
\ee
holds for all $m,n \in \mathbb{N}$ and  for all $x \in X$. 
%Let us also denote by $A_n (x) = A(F^n x)$ and, given for $m<n $, by  
%\bes A^{(m,n)}=  A_m A_{m+1} \dots A_n
%\ees
%For $m<n $, let us denote by  
%\bes A^{(m,n)}=  A_m A_{m+1} \dots A_n .\ees
If $F$ is invertible, let us set $A_{-n}(x) = A(F^{-n}x)$. The map $A^{-1}(x) = A(x)^{-1}$ gives a cocycle over $F^{-1}$ which we call \emph{inverse cocycle}.

%$A_{-n}(x) =A( F^{-n}x)^{-1}$ and 
%$A^{(-n)}(x)= A^{-1}(F^{-1}x) \cdots A^{-1}(F^{-n}x)$ so that (\ref{cocycleid}) holds for all $n,m \in \mathbb{Z}$. Remark that $A^{(-n)}(x)= (A^{(n)}(T^{-n}x))^{-1}$. 
\begin{comment} Remark that $B(x) = A_{-1}(Fx)$ is a cocycle over $(X, \mu, F^{-1})$.
than $A$ is also a cocycle over $(X, \mu, F^{-1})$. Let us set $A_{-n}(x) =A( F^{-n}x)$.
for $n<0$ we can set $A^{(-n)}(x)= A^{-1}(F^{-1}x) \cdots A^{-1}(F^{-n}x)$,  
%(A^{(n)}(F^{-n}x))^{-1}$
 so that (\ref{cocycleid}) holds for all $n,m \in \mathbb{Z}$. Remark that $A^{(-n)}(x)= (A^{(n)}(T^{-n}x))^{-1}$. The \emph{inverse transpose cocycle} $(A^{-1})^T$ is defined by $(A^{-1})^T(x) = (A^{-1})^T(x )$ where $M^T$ denotes the transpose of $M$.
% \begin{rem}\label{invtrasp} $(A^{-1})^T$ is a cocycle over $F^{-1}$. 
%\end{rem}
\end{comment}
% \begin{rem}\label{invtrasp} $(A^{-1})^T$ is a cocycle over $F^{-1}$. 
%\end{rem}
If $Y \subset X$ is a measurable subset, the \emph{induced cocycle} $A_Y$ on $Y$ is a cocycle over $(Y, \mu_Y , F_Y)$ where $F_Y$ is the induced map of $F$ on $Y$ and $\mu_Y = \mu/\mu(Y)$  and $A_Y(y)$ is defined for all $y\in Y$ which return to $Y$ and is given by
\bes
A_Y(y) = A%\left
(F^{r_Y(y) -1 } y
%\right
)
\cdots A\left(F y \right) A\left(y\right) ,
\ees
where $r_Y(y)= \min \{ r \st F^r y \in Y\} $ is the first return time. 
The induced cocycle is an acceleration of the original  cocycle, i.e. if $\{n_k\}_{k\in \mathbb{N}} $ is the infinite sequence of return times of some $y\in Y$ to $Y$ (i.e. $T^{n} y \in Y$ iff $n=n_k$ for some $k\in \mathbb{N}$ and $n_{k+1}> n_k$) then 
\be\label{acceleration}
(A_Y)_k(y) =A_{n_{k+1}-1 }(y)   \cdots   A_{n_{k}+1}(y) A_{n_k}(y) .   % A^{(n_k, n_{k+1})}(x).
\ee
We say that $x \in X$ is \emph{recurrent} to $Y$ if there exists an infinite increasing sequence $\{n_k\}_{k\in \mathbb{N}} $ such that $T^{n_k} x \in Y$. Let us extend the definition of the induced cocycle $A_Y$ to all $x \in X$ recurrent to $Y$. 
If the sequence  $\{n_k\}_{k\in \mathbb{N}} $ is increasing and  contains all $n\in \mathbb{N}^+$ such that  $T^{n} x \in Y$, let us say that $x $ \emph{recurs to} $Y$ \emph{along}  $\{n_k\}_{k\in \mathbb{N}} $. In this case,  let us set
% If  $x $  recurs to} $Y$ \emph{along}  $\{n_k\}_{k\in \mathbb{N}} $.
\bes
A_Y(x) :=  A( y ) A^{n_0}_F (x) , \quad \mathrm{where}\,\,  y:= F^{n_0} x \in Y;  \qquad (A_Y)_{n}(x) :=  (A_Y)_n (y ), \qquad  \mathrm{for}\,\, 
% y:= F^{n_0} x \in Y, \,
  n \in \mathbb{N}^+. 
\ees
If $F$ is ergodic, $\mu$-a.e.~$x \in X$ is recurrent to $Y$ and hence $A_Y$ is defined on a full measure set of $ X$.
%For convenience of notation, let us con

In the rest of the paper, we will use the norm  $\norm{A}  = \sum _{ij} |A_{ij}|$  on matrices (more in general, the same results on cocycles hold for any norm on $SL(d,\mathbb{Z})$). Remark that with this choice $\norm{ A } = \norm{ A ^T }$.  
%We will use $\| A \| = \sum _{ij} |A_{ij}|$. 
A cocycle over $(X , F, \mu)$ is called \emph{integrable} if $\int_X \ln \|A(x)  \| \ud \mu(x) < \infty$. Integrability is the assumption which allows to apply Oseledets Theorem. 
Let us recall the following properties of integrable cocycles\footnote{The integrability of the dual cocycle stated in Remark \ref{inducedintegrable} $(i)$ is proved in Zorich \cite{Zo:dev} and since $\norm{A^{-1}} = \norm{(A^{-1})^{T}}$, the integrability of the inverse cocycle follows. The proof of  Remark \ref{inducedintegrable} $(ii)$  if $F$ is invertible follows from the Kac's lemma representation of the space as towers and the non-invertible case can be reduced to the invertible one by considering the natural extension of $F$.}. 
\begin{rem}\label{inducedintegrable}
If $A$ is an integrable cocycle over $(X , F, \mu)$ assuming values in $SL(d,\mathbb{Z})$, then
\begin{itemize} \item[(i)] The dual cocycle $(A^{-1})^{T}$ and, if $F$ is invertible,  the inverse cocycle $A^{-1}$ over  $(X , F^{-1}, \mu)$  are integrable;
\item[(ii)] Any induced cocycle $A_Y$ of $A$ on a measurable subset $Y \subset X$ is integrable. 
\end{itemize}
\end{rem}
%\footnote{The integrability of the dual cocycle in $(i)$ is proved in Zorich \cite{Zo:dev} and since $\norm{A^{-1}} = \norm{A^{-1})^{T}}$, the integrability of the inverse cocycle follows. The proof of $(ii)$ if $F$ is invertible follows from the Kac's lemma representation of the space as towers and the non-invertible case can be reduced to the invertible one considering the natural extension of $F$.}
%%%A%%% Add Reference
\begin{comment}
\begin{proof} To prove $(ii)$, if $F$ is invertible, it is enough to represent $X$ as towers over $Y = \cup_{n\geq 0 } Y_n$, where $Y_n$ is the set where $r_Y=n$m, i.e. $X= \cup_{n\geq 0 }\cup_{k=0}^{n-1}Y_n$ to check that $\int_Y  \ln\| A_Y(y)\| \ud \mu_Y \leq \mu(Y) \int_X \ln \| A \|$. 
In case $F$ is not invertible, one can reduce to the previous case by considering the natural extension $\widetilde{F}$ on $\widetilde{X}$ and extend the cocycle by setting $A(\widetilde{x})= A(\pi(\widetilde (x) )$ where $\pi$ is the projection of $\widetilde{X}$ to $X$. 
\end{proof}
\end{comment}
%Definition of induced  cocycle 
%An accelerated cocycle of an integrable cocycle is integrable. 
%Inverse cocycle, integrable too
%Natural extension and inverse cocycle in the past
In \S \ref{Rauzysec} we will consider the Rauzy-Veech Zorich cocycle for IETs and in \S \ref{absencesec} we will use various accelerations constructed using the following two Lemmas. For $m<n $, let us denote by\footnote{The reader should remark that here the order of the matrices in the product is the inverse order than the one used in (\ref{cocycleid}). This notation is convenient since we will apply it to matrices  $Z$ where $Z^{-1}$ is the Rauzy cocycle.}  
\bes A^{(m,n)}=  A_m A_{m+1} \dots A_{n-1} .
\ees
%$A^{(m,n)}=  A_m A_{m+1} \dots A_n$ 
\begin{lemma}\label{expgrowthlemma}
Assume $A^{-1}$ be an integrable cocycle over an ergodic and invertible $(X,\mu, F)$. %For each $\epsilon>0$ 
There exists a measurable $E_1 \subset X$  with positive measure\footnote{The same proof gives that for each $\epsilon>0$ there exists $E_1$ with $\mu(E_1) > 1-\epsilon$.}  and a constant $\overline{C}_1>0$ such that for all $x\in X$ recurrent to $E_1$ along the sequence $\{n_k \}_{k\in \mathbb{N}}$ we have
% if $\{n_k \}_k\in \mathbb{N}$ is the sequence of 
% visits to $E_1$, such that 
%\begin{itemize}
%\item[(i)]
\be\label{expgrowth}
\frac{ \ln \| A^{(n, n_k ) } (x) \|}{n_k - n} \leq \overline{C}_1, \qquad  \forall \, 0 \leq n < n_k .
%, \qquad \forall n < n_k .
\ee 
%\item[(ii)]
%$\frac{ \ln \| (A_E)_{ k - m  } (x) \|}{n_k - n} \leq C$ for all $n < n_k$;
%\end{itemize}
\end{lemma}
\begin{proof}
Since $A^{-1}$ is integrable, by Remark \ref{inducedintegrable} $(i)$, also the inverse cocycle $A$ over $(X, \mu, F^{-1})$  is integrable. Hence, by Oseledets theorem, the functions $\ln \| A_{F^{-1}}^{ m } \| /  m  $ converge pointwise. There exists a set $E_1$ of positive measure such that by Egorov theorem the convergence is uniform, so that $\ln \| A_{F^{-1}}^{m }(x) \| \leq c  m  $ for some $c>0$ and all $x\in E_1$ and all $m\geq \overline{m}>0$  and at the same time $\norm{ A_{- m } (x)} $ for $0\leq m < \overline{m}$ are uniformly bounded. % Moreover, by ergodicity a.e. $x\in X$ visits $E_1$ infinitely often. 
Thus, if $F^{n_k} x \in E$, we have $\ln \norm{A_{F^{-1}}^{m }\left(F^{n_k} x\right)} \leq C m$ for some $C>0$ and all $m\geq 0$. Hence, since $A_{F^{-1}}^{m }\left(F^{n_k} x\right)= A^{(n_k - m +1,n_k+1)}(x) $, changing indexes by $n=n_k-m+1$, we get  (\ref{expgrowth}).
\begin{comment}
there exists $\theta$ such that for a.e. $x \in X$,  $\ln \| A^{(- m )} \| \leq m \theta $ for all $m$ large enough. Thus there  exists a measurable $c(x)$ such that $\ln \| (A^{(-m)} (x) \| \leq  c(x) {m \theta} $ for all $m \in \mathbb{N}^+$. On a set $E$ of positive measure, $c(x) \leq c$ for some $c>0$. Thus, if $F^{n_k} x \in E$, using that $A^{(-m )}\left(F^{n_k} x\right)= A^{(n_k - m +1 ,n_k+1)}(x) $ and changing indexes with $m=n_k-n+1$, we have (\ref{expgrowth}). 
\bes%\label{expgrowth}
 \ln \| A^{(n_k-n, n_k ) } (x) \|} \leq c \theta { n} , \qquad \forall n \in \mathbb{N}^+,
\ees 
which gives (\ref{expgrowth}) with $C=c\theta$ by changing indexes. 
\end{comment}
\end{proof}
\begin{lemma}\label{subexpgrowthlemma}
Under the same assumptions of Lemma \ref{expgrowthlemma}, for each $\epsilon>0$ there exists a  measurable $E_2 \subset X$ with %$\mu(E_2) > 1-\epsilon$ 
positive measure 
and a constant $\overline{C}_2>0$ such that if $x\in X$ is recurrent to $E_2$ along the sequence $\{m_k \}_{k\in \mathbb{N}}$, we have 
%$A_{E_2}$ denoted the induced cocycle, 
%\begin{itemize}
%\item[(i)]
\be\label{subexpgrowth}
 \| A_{m_k - n  } (x) \|  \leq \overline{C}_2 e^{\epsilon n }, \qquad \forall \,  0\leq n \leq m_k. %in \mathbb{N}^+. 
 \ee 
%\be\label{subexpgrowth}
% \| (A_{E_2})_{(k - m ) } (x) \|  \leq C_2 e^{\epsilon m },
%, \qquad \forall n < n_k .\ee 
%\item[(ii)]
%$\frac{ \ln \| (A_E)_{ k - m  } (x) \|}{n_k - n} \leq C$ for all $n < n_k$;
%\end{itemize}
\end{lemma}
\begin{proof}
Recall that since $F^{-1}$ is ergodic, if $f$ is integrable, the functions $\{ f\circ F^{-m}/m  \}_{m\in \mathbb{N}}$ converge to zero   for a.e.~$x \in X$ %as a sequence of $m $
 and hence, by Egorov theorem, are eventually uniformly less than $\epsilon$ on some positive measure set for $m\geq \overline{m}$. Since $A^{-1}$ is integrable, also $A$ is integrable (Remark \ref{inducedintegrable} $(i)$) and applying this observation to $f= \ln \| A \| $, we can find a smaller positive measure set $E_2$ and $\overline{C}_2>0$ (in order to bound also $ \norm{ A ( F^{-m} x ) }$ for $x \in E_2$,  $0\leq m \leq \overline{m}$) such that if $y \in E_2$, we have $ \norm{ A (F^{-n}  y)} \leq \overline{C}_2 e^{\epsilon n} $ for all $n \geq 0$. When $y= F^{m_k}x \in E_2$, this gives (\ref{subexpgrowth}). 
 % for  and as the sequence $\{m_k\}_{k\in \mathbb{N}}$ of visits to $E_2$ exists for a.e. $x$ by ergodicity. 
\end{proof}

\subsection{Rauzy-Veech-Zorich cocycle}\label{Rauzysec}
\paragraph{Rauzy-Veech and Zorich algorithms.}
The Rauzy-Veech algorithm and the associated cocycle were originally introduced and developed in the works by Rauzy  and Veech \cite{Ra:ech, Ve:int, Ve:gau} and proved since then to be a powerful tool to study interval exchange transformations. 
If $T=(\underline{\lambda}, \pi)$  satisfies  Keane's condition recalled in \S\ref{IETsdefsec}, which holds by \cite{Ke:int}) for a.e.~IET,  the Rauzy-Veech algorithm produces a sequence of IETs  which are induced maps of $T$ onto a sequence of nested subintervals contained in $I^{(0)}$. The intervals are chosen so that the induced maps are again a IETs of the same number $d$ of exchanged intervals. 
For the precise definition of the algorithm, we refer e.g. to the recent lecture notes by Yoccoz \cite{Yo:con} or Viana \cite{Vi:IET}. We recall here only some basic definitions and properties needed in the proofs.

Let us use, here and in the rest of the paper, the vector norm $|\underline{\lambda}|=\sum_{i=1}^{d}\lambda_i$. If $I'\subset I^{(0)}$ is the subinterval associated to one step of the algorithm and $T'$ is the corresponding induced IET, the \emph{ Rauzy-Veech map}  $\R $ associates to $T$ the IET $\R (T)$  obtained by renormalizing $T'$ by $\Leb{I'}$ so that the renormalized IET is again defined on an unit interval.
%after one step of the algorithm including the following renormalization:
%\bes \R \left ( (\underline{\lambda}^{(0)},\pi^{(0)}) \right) \doteqdot \left( \frac{\underline{\lambda}^{(1)}}{|\underline{\lambda}^{(1)}|},\pi^{(1)} \right).
%\ees 
%Let us call it the \emph{Rauzy-Veech map}. 
The natural domain of definition of the map $\R$ is a full Lebesgue measure subset of the space $X:= \Delta_{d-1} \times \R(\pi)$ where $\R(\pi) $ is the Rauzy class\footnote{Let us recall that the Rauzy class of $\pi$ is the subset  of all permutations $\pi'$ of $d$ symbols which appear as permutations  of an IET $T'=(\underline{\lambda}', \pi')$ in the orbit under $\R$ of some IET $(\underline{\lambda}', \pi)$  with initial permutation $\pi$.} of the permutation $\pi$.

Veech proved in \cite{Ve:gau} that $\R$ admits an invariant measure $\mu_{\V}$ which %. % with respect to which it is conservative. The measure $\mu_{\V}$, which  
 is absolutely continuous with respect to the Lebesgue measure, but infinite. 
%The main result proved by Veech  is that the. %The main result proved by Veech in \cite{Ve:gau} is that the map $\R$ is conservative. %As a consequence, he proves that, given any irreducible $\pi \in \Sym{d}$, for a.e. $\underline{\lambda} \in \Delta_{d-1}$, the IET $T=(\underline{\lambda}, \pi)$ is uniquely ergodic.
Zorich showed in \cite{Zo:fin} that one can accelerate\footnote{The acceleration of a map is obtained a.e.-defining  an integer valued function $z(T)$ which gives the return time to an appropriate section. The accelerated map is then given by $\Z (T) := \R^{z(T)} (T )$.} the map  $\R $ in order to obtain a map $\Z$, which we call \emph{Zorich map}, that admits a  \emph{finite} invariant measure $\mu_{\Z}$. Let us also recall that both $\R$ and its acceleration $\Z$ are \emph{ergodic} %\footnote{Since $\R$ preserves an infinite invariant measure, ergodicity means here that if $A$}
 with respect to $\mu_{\R}$ and $\mu_{\Z}$ respectively \cite{Ve:gau}.
 Let us recall the definition of the cocycle associated by the algorithm to the map $\Z$.
% and the associated cocycle (see \ref{Zorich}).  

\begin{comment}
\subsubsection{Zorich acceleration.}
Take an IET $T$ and consider its Rauzy-Veech orbit $\{\R^n T \}_{n\in \mathbb{N}}$.
In a typical situation one can find an integer $z_o=z_o(T)>0$ so that $T, \R T,  \dots , \R^{z_0-1}(T) $ all correspond to the same case $(a)$ or $(b)$ while $\R^{z_o}(T) $ corresponds to the other one. Grouping together these $z_0$ steps of Rauzy induction, we get a new transformation $\Z$ on the space of IET, where the letter $\Z$ is chosen in honor of A. Zorich who introduced this map in \cite{Zo:fin}. %From the point of view of general ergodic theory $\Z$ is the result of changing velocity of $\R$. 

\end{comment}

\paragraph{Lengths-cocycle.}
Consider the Zorich map $\Z$ on $X= \Delta_{d-1} \times \R(\pi)$. We denote by $\{I^{(n)}\}_{n\in\mathbb{N}}$ the sequence of inducing intervals corresponding to the Zorich acceleration of the Rauzy-Veech algorithm (well defined if $T$ satisfies the Keane's condition). Let $T^{(n)}= \Z^n (T)$ be the renormalized induced IET, which is given by $T^{(n)}:=(\pi^{(n)}, \underline{\lambda}^{(n)}/\lambda^{(n)})$, where $\lambda^{(n)}= | \underline{\lambda}^{(n)}|= Leb(I^{(n)})$. For each $T= T^{(0)}$ for which $\Z (T) = (\pi^{(1)}, \underline{\lambda}^{(1)}/\lambda^{(1)} )$ is defined, let us associate to $T$ the matrix $Z=Z(T)$ in $SL(d, \mathbb{Z})$  such that $\underline{\lambda}^{(0)}= Z \cdot  \underline{\lambda}^{(1)}$. The map  $Z^{-1}$: $X \rightarrow SL(d,\mathbb{Z})$ is a cocycle over $(X, \mu_{\Z},  \Z ) $ , which we call the \emph{Zorich lengths-cocycle}. Zorich proved in \cite{Zo:fin} that $Z^{-1}$  is \emph{integrable}. 

Defining  $\RL{n} = \RL{n} (T) \doteqdot  Z(\Z^n (T))$ and $\RLp{n} \doteqdot \RL{0} \cdot \dots \cdot \RL{n-1}$  and iterating the lengths relation, we get   
%Then for each $r$ we can associate to $T$ the product:
%\bes 
%\ees
\begin{equation} \label{lengthsrelation}
 \underline{\lambda} ^{(n)} = \left( \RLp{n} \right) ^{-1 }  \underline{\lambda}, \qquad \mathrm{where} %$\underline{\lambda}^{(r)}$ \, \, is \st  
 \quad \Z ^n (T) := \left(\frac{\underline{\lambda}^{(n)}}{\lambda^{(n)}}, \pi^{(n)}\right) .
\end{equation}
For more general products with $m<n$ we use the notation $Z^{(m,n)} \doteqdot \RL{m}   \RL{m+1}  \, \dots \, %\cdot \RL{n-2}
  \RL{n-1}$.
By our choice of the norm $|\underline{\lambda}|= \sum_{i}|\lambda_i|$ on vectors and $\norm{A} = \sum_{i,j}|A_{ij}| $ on matrices, from (\ref{lengthsrelation}), % we have 
\be\label{normproduct}
| \underline{\lambda}^{(m)} | = \lambda^{(m)}    \leq \norm{ Z^{(m,n)} } \lambda^{(n)} .  
\ee
Moreover, if $Z^{(n,m)}= A_1 \cdots A_N$ where each of the matrices $A_i$ has strictly positive entries, then
\be \label{positvegrowth}
{\lambda}^{(m)} \geq d^{N}  \lambda^{(n)} .  
\ee

%Let us also recall that one can define the natural extension $\hat{\Z}$ of the map $\Z$ (see \cite{Vi:IET, Yo:con}), whose domain admits a geometric interpretation in terms of the space of zippered rectangles introduced by Veech \cite{Ve:gau}. The cocycle $Z^{-1}$ can be extended to the invertible map $\hat{\Z}$ to be locally constant on fibers of the natural projection. 

%\paragraph{Cocycle over the natural extension.}
\label{natextsec}
The \emph{natural extension} $\hat{\Z}$ of the map $\Z$ is an invertible map defined on a domain $\hat{X}$
such that there exists a projection $p : \hat{X} \rightarrow X$ for which $p \hat{\Z} = \Z p $ (see \cite{Yo:con, Vi:IET} for the explicit definition of $\hat{X}$, which admits a geometric interpretation in terms of the space of zippered rectangles). %$$ wich semiconjugates $\hat{Z}$ and $\Z$, i.e.~
 The natural extension $\hat{\Z}$ preserves a natural invariant measure ${\mu}_{\hat{\Z}}$, which gives $\mu_{\Z}$ as pull back by $p$. % = p\*\hat{m}$.  
%  The domain $\hat{X_}_{\pi}$ of the natural  extension 
The cocycle $Z^{-1}$ can be extended to a cocycle over $(\hat{X},  {\mu}_{\hat{\Z}},\hat{\Z})$ by defining the extended cocycle, for which we will use the same notation $Z^{-1}$, to be  constant on the fibers of  $p$.

\paragraph{Towers and induced partitions}
The action of the initial interval exchange $T$ can be seen in terms of Rohlin towers over $T^{(n)}:=\Z^n(T)$ 
 as follows. Let $\underline{h}^{(n)} \in \mathbb{N}^d$ be the vector such that $h^{(n)}_j$
  gives the return time of any $x\in I^{(n)}_j$ to $I^{(n)}$.
Define the sets 
%%% ACCORCIABILE
\bes
Z^{(n)}_j \doteqdot \bigcup _{l=0}^{h^{(n)}_j-1} T^l I^{(n)}_j.
\ees
%When $T$ is ergodic, $\bigcup_{j=1}^{d} Z^{(r)}_j$ is a non-trivial $T$-invariant set, therefore the sets $Z^{(r)}_j$, $1\leq j\leq d$ give a partition of the whole $I$. 
Each $Z^{(n)}_j$ can be visualized as a tower over $I^{(n)}_j\subset I^{(n)}$, of height  $h^{(n)}_j$, %(see Figure \ref{Deltak}, page \pageref{Deltak}), 
whose floors are $T^l I^{(n)}_j$.  
%. A floor of the tower is an interval of the form 
%$Z^{(n)}_{j,l} \doteqdot  for some 
%$l=0, \dots, h^{(r)}_j-1$.
Under the action of $T$ every floor but the top one, i.e.~if $0\leq l <h^{(n)}_j\!-\!1$, moves one step up, while the image by $T$ of the last one (corresponding to $l=h^{(n)}_j\!-\!1$) is $ T^{(n)} I^{(n)}_j $.

Let us denote by $\phi^{(n)}$ be the partition of $I^{(0)}$ into floors of step $n$, i.e.~intervals of the form $T^{l} I^{(n)}_j$. We say that  $F \in \phi^{(n)}$ is of \emph{type} $j$, where $1\leq j\leq d$, if it is a floor  of $Z^{(n)}_{j}$. 
The following well-known fact is proved for example in \cite{Yo:con, Vi:IET}. 
\begin{rem}\label{minimalrem}
If $T$ satisfies the Keane's condition, the partitions $\phi^{(n)}$ converge as $n$ tends to infinity to the trivial partitions into points.
\end{rem}
\noindent
We recall also that the entry $Z^{(n)}_{ij}$ of the matrix $Z^{(n)}$ equals to the number of visits of the orbit of any point $x\in I^{(n)}_j$ to the interval $I^{(0)}_i$ of the original partition before its first return to $I^{(n)}$.
%We recall that  have a \emph{dynamical meaning} in terms of return times. Namely, denote $I^{(r)}_j$, $1\leq j\leq d$, the subintervals of $T^{(r)}$. 
%\begin{rem} \label{dynmean}
%The entry $\RLp{r}_{ij}$ is 
%\end{rem}
%Therefore, the norm $h^{(r)}_j$ of the $j^{th}$ column of $\RLp{r}$, i.e. $h^{(r)}_j \doteqdot \sum_{i=1}^{d} \RLp{r}_{ij}$ %gives the return time of any $x\in I^{(r)}_j$ to $I^{(r)}$.
Moreover, if $\underline{h}^{(0)}$ is the column vector with all entries equal to $1$,  the height vectors $\underline{h}^{(n)}$ can be obtained by applying the dual cocycle, i.e. 
\be\label{heightsrelation}
\underline{h}^{(n)} = (Z^T)^{(n)} \underline{h}^{(0)}.
\ee
%%% ACCORCIABILE
%iterating the transpose cocycle $Z^T$ to the initial vector with all entries one.

\paragraph{Balanced return times.}\label{bndtypesec} 
%\begin{defn}
Consider an orbit $\{\Z^n (T) \}_{n\in \mathbb{N}}$ of a $T$ satisfying Keane's condition. 
Let us say that a sequence $\{n_l\}_{l\in \mathbb{N}}$ is a sequence of \emph{balanced times} for $T$  if there exists  $\nu>1$
 such that the following hold for all $l \in \mathbb{N}$:
\be
\label{balance} 
 \frac{1}{\nu} \leq \frac{\lambda_i^{(n_l)}}{\lambda_j^{(n_l)}} \leq \nu, \qquad 
%\label{heightsbalance} 
\frac{1}{\nu} \leq \frac{h_i^{(n_l)}}{h_j^{(n_l)}} \leq \nu, \quad  \forall \, 1 \leq i, \, j \leq d. \quad
\ee
%\begin{eqnarray}
%\label{lengthsbalance} 
%&{\it 1.}&\quad \quad \frac{1}{\nu} \leq \frac{\lambda_i^{(n_l)}}{\lambda_j^{(n_l)}} \leq \nu, \quad  \forall \, 1 \leq i, \, j \leq d;  \qquad  \emph{(balance of  lengths)}
%\\
%\ee
%\item % \emph{($\kappa$ -balance of  heights)} \be 
%&{\it 2.}&\quad \quad 
%\label{heightsbalance} 
%\frac{1}{\nu} \leq \frac{h_i^{(n_l)}}{h_j^{(n_l)}} \leq \nu, \quad  \forall \, 1 \leq i, \, j \leq d; \quad \emph{(balance of  heights)}
%\end{eqnarray}
%\end{defn}
If $n$ is such that the tower rapresentation over $\Z^n (T)$ satisfies (\ref{balance}), we call $n$ a \emph{balanced return time}.    
Lengths and heights of the induction towers
are approximately of the same size if $n$ is a  balanced return time or, more precisely:
%\begin{rem}
%If $n$ is a balanced return time, for each $j=0,\dots, d$,
\be \label{balancedcomparisons} 
 \frac{1}{d\nu }{\lambda^{(n)} } \leq  \lambda^{(n)}_j \leq \lambda^{(n)}  ; \qquad  
\frac{1}{\nu\lambda^{(n)} }  \leq  h^{(n)}_j \leq \frac{\nu}{ \lambda^{(n)} } , \qquad \forall \, j=0,\dots, d. %
\ee

\paragraph{Hilbert metric and projective contractions.} \label{Hilbertsec}
Consider on the simplex $\Delta_{d-1} \subset \mathbb{R}_+^d$ the \emph{Hilbert distance} $d_H$, defined as follows.
%%% ACCORCIABILE?
\bes
d_H(\underline{\lambda}, \underline{\lambda}') \doteqdot \log \left(  \frac{\max_{i=1,\dots , d} \frac{\lambda_i}{\lambda'_i}}{\min_{i=1 , \dots ,d}\frac{\lambda_i}{\lambda'_i}}\right).
\ees

%We denote the diameter with respect to $d_H$ of a projective subset $\Lambda \subset \Delta_{d-1}$ by
%\be \label{diameter}
%diam_H (\Lambda ) \doteqdot \sup_{\lambda, \lambda' \in \Lambda} d_H( \lambda, \lambda' ) .
%\ee
%Remark that if its closure $\overline{\Lambda} \subset \Delta_{d-1}$, 
%  then $diam_H (\Lambda )$ is finite. 

Let us write $A\geq 0$ if $A$ has non negative entries and $A>0$ is $A$ has strictly positive entries.
Recall that to each $A\in SL(d,\mathbb{Z})$, $A\geq 0$, one can associate a projective transformation $\widetilde{A}: \Delta_{d-1} \rightarrow \Delta_{d-1}$ given by
%%% ACCORCIABILE?
\bes
\widetilde{A} \underline{\lambda} = \frac{A \underline{\lambda} }{|A \underline{\lambda} |}.
\ees
When $A \geq 0$, $d_H(\widetilde{A} \underline{\lambda}, \widetilde{A} \underline{\lambda}' ) \leq d_H(\underline{\lambda}, \underline{\lambda}')$. Furthermore, if $A>0$, then we get a contraction. More precisely, if $A>0$, since the closure $\widetilde{A}  \left( {\Delta}_{d-1} \right)$ is contained in $\Delta_{d-1}$, we have
%is equivalent to the closure $ \overline{ \widetilde{A} \left( \Delta_{d-1} \right)}$ being contained in $\Delta_{d-1}$, hence defining 
\be \label{contraction}
 d_H(\widetilde{A} \underline{\lambda}, \widetilde{A} \underline{\lambda}' ) \leq (1-e^{-D(A)}) d_H(\underline{\lambda}, \underline{\lambda}'), \qquad \mathrm{where} \quad 
 D(A) \doteqdot \sup_{\underline{\lambda}, \underline{\lambda}' \in \Delta_{d-1}}d_H( \widetilde{A} \underline{\lambda}, \widetilde{A} \underline{\lambda}' ) < \infty. 
\ee

\begin{comment}
\begin{eqnarray} &&  \label{D_Adef}
D(A) \doteqdot diam_H \left( \, \overline{ \widetilde{A}  \left( {\Delta}_{d-1} \right) }\, \right) =  \sup_{\widetilde{A}\lambda, \widetilde{A} \lambda' \in \Lambda} d_H( \lambda, \lambda' ) < \infty, 
\\
&& \label{contraction}
 d_H(\widetilde{A} \lambda, \widetilde{A} \lambda' ) \leq (1-e^{-D(A)}) d_H(\lambda, \lambda').
\end{eqnarray}
\end{comment}

\section{Rigidity sets and Kochergin criterion}%\label{rigiditysec}
\subsection{A  condition for absence of mixing.}\label{criteriumsec}
\paragraph{Rigidity sets}
Interval exchange transformations present some type of rigidity, which was used by Katok in \cite{Ka:int} to show that they are never mixing. Let us formalize it in the following definition.
\begin{defn}[Rigidity sets and times]\label{rigiditysetsdef} The sequence $\{E_k\}_{k\in \mathbb{N}}$ of measurable subsets $E_k \subset I$ form a sequence of \emph{rigidity sets} if there exists a corresponding increasing sequence of \emph{rigidity times} $\{ r_k\}_{k\in \mathbb{N}}$, $r_k\in \mathbb{N_+}$, and a sequence of  finite partitions $\{\xi_{k}\}_{k\in \mathbb{N}}$ 
%such that $\mesh$
converging to the trivial partition into points
 and a constant $\alpha>0$ such that
\begin{itemize}
\item[(i)] $Leb(E_k)\geq \alpha$ for all $k \in \mathbb{N}$;
 \item[(ii)] for any $F \in \xi_{k}$, $T^{r_k} (F\cap E_k )\subset F$.
\end{itemize}
\end{defn}
\noindent Condition $(ii)$ is a way to express that $T^{r_k}$ is close to identity on $E_k$.
%If we consider a rotation $R_{\alpha}$, if $q_k$ are the denominators of the continued fraction expansion of $\alpha$, it is well known that $R_{\alpha}^{q_k}$ converges to  the identity map as $k$ tends to infinity.  The main reason while interval exchange transformations are never mixing (as shown by Katok, \cite{Ka:int}) is that they also present a certain rigidity. The rigidity, though, is not present   on the whole interval, but only on certain positive measure subsets, which motivates the following definition.

In order to show absence of mixing for a special flow whose base presents this type of rigidity, it is enough to verify the following criterion, which was first used and proved  by Kochergin  in \cite{Ko:non}.
\begin{lemma}[Absence of mixing criterion] \label{absencemixingcriterion}
If there exist %a sequence of finite partitions $\{\xi_{k}\}_{k\in \mathbb{N}}$ converging to the trivial partition into points, a sequence $\{r_k\}_{k\in \mathbb{N}}$ of natural numbers, 
a sequence $\{ E_k  \}_{k\in \mathbb{N}}$ of rigidity sets $E_k \subset I$ with corresponding rigidity times $\{r_k\}_{k\in \mathbb{N}}$ %and $Leb(E_k)\geq \alpha$
 %$\{ E_k \subset I \}_{k\in \mathbb{N}}$ %($C_k\subset [0,1[$)
 and a constant $ M> 0$ such that%, for all $k\in \mathbb{N}$,
\begin{itemize}
%\item[(i)] $Leb(E_k)\geq \alpha$;
% \item[(ii)] for any $F \in \xi^{k}$, $T^{r_k} (F\cap E_k )\subset F$;
\item[(iii)] for all $k\in \mathbb{N}$, for all $y_1,y_2 \in E_k$, $|\BS{f}{r_k}(y_1)- \BS{f}{r_k}(y_2)| < M$,
\end{itemize}
then the special flow $\{\varphi_t\}_{t\in \mathbb{R}}$ is not mixing. 
\end{lemma}
Condition $(iii)$ is described sometimes saying that Birkhoff sums $\BS{f}{r_k}$ ``do not stretch''. Stretching of Birkhoff sums is the main mechanism which produces mixing in special flows over rotations or over interval exchange transformations when the roof function has logarithmic asymmetric singularities (see e.g. \cite{SK:mix, Ul:mix}). Lemma \ref{absencemixingcriterion} show that, when there is rigidity in the base, stretching of the Birkhoff sums is also a necessary condition to produce mixing.

We use Lemma \ref{absencemixingcriterion} to prove Theorem \ref{absencethm}. In $\S \ref{rigiditysec}$ we describe  the construction of a class of sequences of rigidity sets  $E_k$ and times $r_k$ for typical IETs, which are used in the proof of Theorem \ref{absencethm}.  The sets that we construct are analogous to the type of sets used by Katok in \cite{Ka:int} to show that IETs are never mixing, but are constructed with the help of Rohlin towers for Rauzy-Veech induction.  
A variation of this construction is used by the author also in \cite{Ul:wm}, for the proof of weak mixing for this class of flows. 
%In the next paragraph, we briefly recall only the facts about Rauzy-Veech induction that we need.

\subsection{Construction of rigidity sets.}\label{rigiditysec}
Assume $T$ satisfies Keane's condition. Let $\{n_k\}_{k\in \mathbb{N}}$ be a sequence of balanced times for $T$. Consider the corresponding towers $Z^{(n_k)}_j$ for $j=1,\dots, d$. % corresponding to the ${n_k}^{th}$ balanced induction time. 
By pigeon hole principle, since $\sum_j h^{(n_k)}_{j} \lambda^{(n_k)}_{j}=1$, we can choose $j_0$ such that
% the area of the tower 
\be \label{bigtower}   h^{(n_k)}_{j_0} \lambda^{(n_k)}_{j_0} \geq \frac{1}{d} . \ee 
\begin{comment}
\begin{figure}
\centering
%\subfigure[The set $E_k$.]{ 
\includegraphics[width=0.3\textwidth]{Ck.eps} 
%\hspace{9mm}
%\subfigure[$T^j({T^i(\Delta'^{(n_k)}_{j_0})_ {l_0}})$, $0\leq i < r_k$.\label{disjointness}]{
%\includegraphics[width=0.33\textwidth]{Deltakimages.eps} }
\caption{\label{Deltak} The rigidity sets $E_k$.}
\end{figure}
\end{comment}
The map $T'$ obtained inducing on $T$ on $I^{(n_k)}_{j_0}$ is an IET of at most $d+2$ intervals (see for example \cite{CFS:erg}), which we denote $(I^{(n_k)}_{j_0})_ l$, where $0\leq l\leq d+2$. Let $h^{(n_k)}_{j_0, l}$ be the first return time of  $(I^{(n_k)}_{j_0})_ l$ to  $I^{(n_k)}_{j_0}$ under $T$. Choose $l_0$ so that 
\be \label{biginterval} \Leb{(I^{(n_k)}_{j_0})_ {l_0}}\geq \frac{1}{d+2}\Leb{I^{(n_k)}_{j_0}} . \ee
% See Figure \ref{Deltak}.
Let $J_k \subset  (I^{(n_k)}_{j_0 } )_{ l_0} $ be any subinterval such that  $\Leb{J_k} \geq \beta \,  \Leb{ (I^{(n_k)}_{j_0} )_{ l_0} } $ for some $0\leq \beta \leq 1$. Define
\be\label{Ck}
 E_k :=  \bigcup_{i=0} ^{h^{(n_k)}_{j_0 }-1} T^i J_k %I^{(n_k)}_{j_0 i , l_0}
; \qquad r_k := h^{(n_k)}_{j_0, l_0},\ee
i.e., $E_k$ %, as shown in Figure \ref{Ck},
 is the part of the tower $Z^{(n_k)}_{j_0}$ which \emph{lie above} $J_k$. % (see Figure \ref{Deltak}). % and $ r_k$ is the corresponding return time of $(I^{(n_k)}_{j_0})_ l$ to . 
Let $\xi_{k}= \phi^{(n_k)}$ be the sequence of partitions into floors corresponding to the considered balanced steps.  

\begin{lemma}\label{rigidityiandii}
Any sequence of $\xi_{k}$, $ r_k$ and $E_k$ defined as above satisfy the assumptions $(i)$ and $(ii)$ of Definition \ref{rigiditysetsdef}. 
\end{lemma}

%\begin{proof}
%Since $T$ is minimal, % as remarked in \ref{Tminmaltrivialxi},
% the partitions  $\xi^{k}$ tend to the trivial partition into points as $k$ tends to infinity.   Property $(i)$  with $\alpha %= \beta/d(d+2)$ follows from (\ref{biginterval}) and  (\ref{bigtower}) and the assumption on $\Leb{J_k}$, since
%$\Leb{E_k} = h^{(n_k)}_{j_0 } \Leb{J_k} \geq \beta \frac{1}{d+2}  h^{(n_k)}_{j_0 } \Leb{I^{(n_k)}_{j_0}} \geq \alpha$.
%To show $(ii)$, let $F\in \xi^k$. Since $E_k \subset  Z^{(n_k)}_{j_0}$, $F\cap E_k =\emptyset$ unless we considered a floor of %$ Z^{(n_k)}_{j_0}$, i.e., $F=  T^i I^{(n_k)}_{j_0} $ for some $0\leq i \leq h^{(n_k)}_{j_0 }$. In this case $F\cap E_k = T^i J_k$ and, since $T^{r_k} J_k \subset I^{(n_k)}_{j_0} $ by definition of $r_k$, $T^{r_k}( F\cap E_k ) =  T^{r_k+i} J_k \subset T^i  I^{(n_k)}_{j_0} = F $.
%In \S\ref{proofabsencethmsec} we will define the  sets $J_k$ so that, when rescricting to an appropriate subsequence of $r_k$, % also $(iii)$ is verifyed.
%\end{proof}

\begin{rem}\label{disjoint}
For any $0\leq j < {h^{(n_k)}_{j_0 }}$,   all $T^i (T^j(I^{(n_k)}_{j_0 } )_{ l_0})$  with $0\leq i < r_k$ are disjoint intervals, which are rigid translates of $(I^{(n_k)}_{j_0 } )_{ l_0}$. The same is true for  $T^i J_{k}$, $0\leq i < r_k$.% intervals and rigid translates of $[a,b[$.
\end{rem}
\noindent 
The proof of Lemma \ref{rigidityiandii} and of Remark \ref{disjoint} can be found in \cite{Ul:wm}. 
%In Figure \ref{disjointness} we show an example of the images  of $T^i ((I^{(n_k)}_{j_0 } )_{ l_0})$ under $T^j$ for $0\leq j < r_k$. 

%%%%%%KEEEP!!!!
\begin{comment}which illustrates the disjointness and the following proof.
\begin{proof}
Since $r_k$ is the first return time of  $(I^{(n_k)}_{j_0 } )_{ l_0}$ to $ I^{(n_k)}_{j_0 } $, $T^i (I^{(n_k)}_{j_0 } )_{ l_0}$  for $0\leq i < r_k$ are all disjoint and contained in $I\backslash  I^{(n_k)}_{j_0 } $ by definition of return time (if $z\in T^{i_1} I^{(n_k)}_{j_0 }\cap  T^{i_2} I^{(n_k)}_{j_0 }$ with $0\leq i_1< i_2< r_k$, then $z=T^{i_1}z_1=T^{i_2}z_2$ and $T^{i_2-i_1}z_2\in I^{(n_k)}_{j_0 } $, contradicting that $r_k$ is the first return for $z_2$).   For $ r_k-j \leq i < r_k$, $T^{i}T^j (I^{(n_k)}_{j_0 } )_{ l_0} \subset T^{i-r_k+j} { I^{(n_k)}_{j_0 }}\subset \cup_{i=0}^{j-1} T^{i} { I^{(n_k)}_{j_0 }}$, so they do not intersect neither  $T^{i}T^j (I^{(n_k)}_{j_0 } )_{ l_0}$ with $0\leq i <h^{(n_k)}_{j_0 }-j$, which are contained in $ T^{i} { I^{(n_k)}_{j_0 }}$ with $j\leq i <h^{(n_k)}_{j_0 }$, nor $T^{i}T^j (I^{(n_k)}_{j_0 } )_{ l_0}$ with $h^{(n_k)}_{j_0 }-j \leq i < r_k-j$, which are contained in $I\backslash Z^{(n_k)}_{j_0 }$. 

Since $J_{k,0}\subset  (I^{(n_k)}_{j_0 } )_{ l_0}$, for each $0\leq j < {h^{(n_k)}_{j_0 }}$, $J_{k,j}\subset T^j (I^{(n_k)}_{j_0 } )_{ l_0}$ and hence also  $T^i J_{k,j}$ with $0\leq i < r_k$ satisfy the same properties.
\end{proof}
\end{comment}

\section{Upper bounds on Bikhoff sums of derivatives.} \label{absencesec} %$|\BS{f'}{r}|$.}%  and absence of mixing.}
%Consider again the functions $u$ and $v$ and assume that $f' = u - v$. 
The key ingredient  to show condition $(iii)$ of the absence of mixing criterion (Lemma \ref{absencemixingcriterion}) 
%show the absence of mixing 
 are upper bounds on the Birkhoff sums  $|\BS{f'}{r_k}| $ on some rigidity set $E_k$ where $r_k$ is the corresponding rigidity times (see definitions in $\S$\ref{criteriumsec}). 
%%Let us introduce two auxiliary functions $u$, $v$ defined on $I^{(0)}$:
%%\bes
%%u(x) := \frac{1}{ x};\qquad v(x) := \frac{1}{1- x}.
%%\ees
%\begin{prop}\label{reducetouv}
%Assume $T$ is uniquely ergodic. There exists a sequence $\alpha_r$ such that $\alpha_r \rightarrow 0$ as $r \rightarrow \infty$ and for all $x$ distinct from singularities of $\BS{f}{r}$,
%\bes
%\BS{f'}{r}(x) = (  - C^+ + \alpha_r^+ )  \BS{u}{r}(x) + (  C^- + \alpha_r^- ) \BS{v}{r}(x), 
%\ees
%where $|\alpha_r^{\pm}| \leq \alpha_r$.
%Assume that $f' = u - v$. Since $T$ is a piecewise isometry, for a.e. $x\in I^{(0)}$, $\BS{f'}{r}(x)=\BS{u}{r}(x) - \BS{v}{r}(x)$.
%In order to estimate Birkhoff  sums $\BS{f'}{r}(z)$, 
Let us first consider Birkhoff sums of the form $\BS{f'}{r_k} (z_0) $ where $z_0 \in I^{(n_k)}_j$ and $r_k= h^{(n_k)}_j$ is exactly the return time.  We call this type of sums  \emph{Birkhoff sums along a tower}, since the orbit segment $\{ T_i z_0 \}_{i=0}^{r_k-1}$ has exactly one point in each floor of the tower $Z^{(n_k)}_j$.

Let us denote by $x_i^{min}$, for  $i=0, \dots, s_1-1$, the minimum distance from the singularity  $\overline{z}^{+}_i$  of the orbit points to the right of $\overline{z}^{+}_i$  and  by $y_i^{min}$,  for $i=0, \dots, s_2-1$,  the minimum distance from the $\overline{z}^{-}_i$  of the points to the left of  $\overline{z}^{-}_i$.  
% from the left and the right,
  In formulas, denoting by $(x)^{pos}$ the positive part of $x$ which is defined by $(x)^{pos}=x$ if $x\geq 0$ and $(x)^{pos}=0$ if $x<0$, these minimum distances are given by
%, denoting by  $|x|^+$  the positive part of $x$ or $|x|^+=x $ if $x>0$ and $0$ otherwise, set
\begin{eqnarray}
x_i^{min} &: =& \min \{ (T^j z_0 - \overline{z}_i )^{pos} , \quad 0\leq j < r_k \}, \qquad  
\quad i=0, \dots, s_1-1;  \label{minx} \\ %\qquad %\min_{j=0}^{r_k-1} \left| T^j x_0 - \overline{x}_i \right\|^+
y_i^{min} &: = &\min \{  (\overline{z}_i  - T^j z_0 )^{pos},  \qquad 0\leq j < r_k \}, \quad i=0, \dots, s_2-1.
%, \quad y_0^{min}:= y_s^{min} . 
\label{miny}
\end{eqnarray}
%\begin{eqnarray}
%x_i^{min} &: =& \min \{ T^j z_0 - \overline{z}_i \st \quad 0\leq j < r_k, \, T^j z_0 >  \overline{z}_i \}, 
%\quad i=0, \dots, s-1;  \label{minx} \\ %\qquad %\min_{j=0}^{r_k-1} \left| T^j x_0 - \overline{x}_i \right\|^+
%y_i^{min} &: = &\min \{  \overline{z}_i  - T^j z_0  \st  \quad 0\leq j < r_k,  \, T^j z_0 <  \overline{z}_i \}, \quad i=1, \dots, s, \quad y_0^{min}:= y_s^{min} . \label{miny}
%\end{eqnarray}
%Clearly, $x_0 = \min_{i=0}^{r_k-1} T^i x_0$. Let us denote by $1-y_0 = \max_{i=0}^{r_k-1} T^i x_0$ the closest point to $1$. 
%For convenience, in what follows, without loss of generality, we will always assume $C=1$, where $C$ is the constant in Definition \ref{symmsingdef}.
\begin{prop}\label{boundSf'growthprop}
% For every irreducible $\pi$ and a.e. $\underline{\lambda}\in \Delta_{d-1}$, 
For a.e.~IET $T$ %, if we consider the Birkhoff sums  $\BS{f'}{r}$ over $T$, %the IET $(\lambda, \pi)$, 
 there exist a constant $M$ and sequence of balanced induction  times $\{c_l\}_{l\in \mathbb{N}}$ such that, if $z_0 \in I^{(c_l)}_j$ and $r_l= h^{(c_l)}_j$, 
\be \label{boundSf'growth}
\left| \BS{f'}{r_l} (z_0) \right| \leq  M r_l + \sum_{i=0}^{s_1-1}  \frac{C_i^+}{x_i^{min}} +   \sum_{i=0}^{s_2-1} \frac{C_i^-}{y_i^{min}}   .
\ee
\end{prop}
\noindent The proof of Proposition \ref{boundSf'growthprop} is given in \S\ref{cancelletionssec}, using the Lemmas proved in  \S\ref{alldeviationssec}. The estimate of Proposition \ref{boundSf'growthprop} for Birkhoff sums along towers is then used in 
%In Propostion \ref{generalsumprop} in
 \S \ref{decompositionsec} to give bounds on more general Birkhoff sums. % that can be decomposed in. % are estimated decomposing them into special Birkhoff sums to which Proposition \ref{boundSf'growthprop} is applyed.
%This proposition provides the key estimate which will allow to verify condition $(iii)$ of the Absence of Mixing criterion (Lemma \ref{absencemixingcriterion}) in \S\ref{proofabsencethmsec}. 

Let us remark that the linear growth in (\ref{boundSf'growth}) is essentially due to a principal value phenomenon of cancellations between symmetric sides of the singularities, which is peculiar of the symmetric case. A similar principal value phenomenon was used, in the case of rotations, in \cite{SU:lim}. In presence of an asymmetric singularity, as shown in \cite{Ul:mix}, $ \BS{f'}{r_l} $ grows as $r_l \log r_l$ on a set of measure tending to $1$ as $l$ tends to infinity. 

\subsection{Deviations estimates.}\label{alldeviationssec}
In order to estimate deviations of ergodic averages, it is standard to first consider deviations for %the quantities %$N^{(m,n)}_{ij}$
%(with $n>m$,  $1\leq i \leq d$) which give 
the number of elements  of $\phi^{(n)}$ of type $j$ inside $I^{(m)}_{i}$, i.e. for the quantities
\bes
N^{(m,n)}_{ij} \doteqdot \# \{ \, h\, | \quad T^h I^{(n)}_{j} \subset I^{(m)}_{i}, \, 0\leq h < h^{(n)}_j \}.
\ees
%\begin{rem}\label{meaningentries}
In terms of the cocycle matrices $N^{(m,n)}_{ij} = Z^{(m,n)} _{ij} $ 
%\end{rem}
%\begin{rem} \label{alltypes} 
and $N^{(m,n)}_{ij}$ gives also the cardinality of elements of $\phi^{(n)}$ of type $j$ inside each element of $\phi^{(m)}$ of type $i$. %, i.e. inside each $T^l I^{(m)}_{i}$ for $0\leq l < h^{(m)}_i$.
%\end{rem}

Let us recall that in \cite{Zo:dev} Zorich proved an asymptotic result on deviations of ergodic averages for characteristic functions of intervals of $\phi^{(0)}$ (hence on the asymptotic growth of $N^{(m,n)}_{ij}$). %In \ref{deviationssec}  % The arguments in \S \ref{deviationssec} aim at a more . 

\subsubsection{Balanced acceleration} \label{balancedaccelerationsec}
Let $Z$ be the Zorich cocycle over the natural extension $\hat{\Z}$ (see \S \ref{natextsec}).  
Let $\hat{K}$ be a compact subset of $\hat{X}$ and  denote by $A:=Z_{\hat{K}}$ be the induced cocycle of $Z$ on ${\hat{K}}$.
If $\hat{T}$ is recurrent to $\hat{K}$, denote by  $\{a_n \}_{n\in \mathbb{N}}$ the sequence of visits of $\hat{T}$ to ${\hat{K}}$.
%We remark that, since $\Z$ is ergodic, for a.e.~$T $ we can choose a lift  $\hat{T}$ of $T$ to $\hat{X}$ (i.e.~$\hat{T}$ such that $p\hat{T}=T$) so that  $\hat{T}$ is recurrent to $\hat{K}$.
%Given $T$, we can choose a lift
%\footnote{For example, to $T=(\underline{\lambda},{\pi})$ (with $\pi$ irreducible) we can associate the suspension data $\underline{\tau}$ where $\tau_i = \pi(i)-i$, $i=1,\dots, d$. Then $(\underline{\lambda},\underline{\tau},{\pi}) \in \hat{X}$ (see \cite{Yo:con, Vi:IET}) and $p(\underline{\lambda},\underline{\tau},{\pi}) = (\underline{\lambda},{\pi})$.}
% $\hat{T}$ of $T$ to $\hat{X}$ such that $p\hat{T}=T$
%  and, if $\hat{T}$ is recurrent to $\hat{K}$, denote by  $\{a_n \}_{n\in \mathbb{N}}$ the sequence of visits of $\hat{T}$ to ${\hat{K}}$. %We remark that $a_n $ for $n\geq \overline{n}= \overline{n}(\hat{K})$ are independent on the choice of the lift.  
  One can choose the compact set\footnote{Using the notation in \cite{AGY:exp}, we can for example choose $\hat{K}$ to be a subset of the Zorich cross-section contained in $\Delta_\gamma \times \Theta_{\gamma}$, where $\gamma$ is a path in the Rauzy class of $\pi$ starting and ending at $\pi$, chosen to be \emph{positive} and \emph{neat} (see \S 3.2.1 and \S 4.1.3 in \cite{AGY:exp} for the corresponding definitions).}
\begin{comment}
Then if $(\lambda, \pi, \tau) \in \hat{K}$ and $h$ are the corresponding return times, $\lambda \in A_{\gamma} \mathbb{R}_=^d$  since $\theta_{\gamma} = A_{\gamma}^{-1} \Theta_{\pi}$ and $\tau= A_{\gamma}^{-1} \tau'$ and the return times $h \in \mathbb{R}_+^{d} $   
\end{comment}  
  % Then $\gamma$ neat guarantees the positivity condition $(i)$ in Lemma \ref{balancedacceleration}.}
   $\hat{K}$ so that, considering the  acceleration corresponding to return times to $\hat{K}$, the following properties hold (the notation is the one introduced in \S \ref{Rauzysec} and more details can be found in \cite{Ul:erg} and \cite{AGY:exp}).
\begin{lemma}\label{balancedaccelerationlemma}
%Let $K$ be any compact subset of $\Delta_{d-1}\times \mathscr{R}(\pi)$ %   the simplex. %, $\overline{\mathscr{A}} \subset \Delta_{d-1} $.
%and let $A:=Z_K$ be the induced cocycle of $Z$ on $K$. Then 
%For each compact set $K$ 
There exists $\overline{D}>0$ and $\nu>1$ depending only on $\hat{K}$ 
 such that %for all $T\in K $,   %$A_n(T):=(Z_K)_n (T)$ satisfy the following properties.
\begin{itemize}
\item[(i)] $A_n =A(\Z^{a_n} T ) > 0$ for each $n \in \mathbb{N}$; 
\item[(ii)] $D(A_n) \doteqdot \sup_{\underline{\lambda}, \underline{\lambda}' \in \Delta_{d-1}} d_H( \widetilde{A}_n \underline{\lambda}, \widetilde{A}_n \underline{\lambda}' ) \leq \overline{D}$;
\item[(iii)] the return times $\{a_n \}_{n\in \mathbb{N}}$ to $\hat{K}$  %in the sequence $\{a_n\}_{n\in \mathbb{N}}$ of visits of $T$ to $ K$ (i.e. $A^{(n)}(T)= Z^{(a_n)}(T)$) 
 are $\nu$-balanced times. 
\end{itemize} 
\end{lemma}

\subsubsection{Deviations estimates for partition intervals.}\label{deviationssec}
Using the balanced acceleration $A$ we can control quantitatively the convergence of $N^{(m,n)}_{ij}$ corresponding to visits $\{a_n\}_{n\in \mathbb{N}}$ to $\hat{K}$. 
%Let us recall that  Zorich in \cite{Zo:dev} proved an asymptotic result on deviations of ergodic averages for characteristic functions of intervals of $\phi^{(0)}$.
%Lemma \ref{ergodicconvergenceN} below gives a quantitative control of the convergence.  % The arguments in \S \ref{deviationssec} aim at a more .
%Let us denote by %$N_A$ the analogous quantities for the acceleration $A$, i.e. let us set
%\bes
%(N_A)^{(n,m)}_{i,j} : = N^{(a_n,a_m)}_{i,j} .
%\ees
\begin{lemma}\label{ergodicconvergenceN}
Let $
(N_A)^{(m,n)}_{ij} : = N^{(a_m,a_n)}_{ij}$. 
There exists $C_{\overline{D}}>0$ such that for each recurrent $\hat{T} \in \hat{K} $, for each pair  $a_m < a_n$ of return times, we have  
\be\label{deviation}
 (N_A)^{(m,n)}_{i j} = \delta^{(a_n)}_j  \, \frac{  \lambda^{(a_m)}_i }{\lambda^{(a_n)}_j} \left( 1+ \epsilon^{(a_m,a_n)}_{i j} \right),  \qquad \left| \epsilon^{(a_m,a_n)}_{i j} \right|\leq 
 C_{\overline{D}} {(1-e^{-\overline{D}})}^{n-m}  ,
\ee
for all $1\leq i,j \leq d$, where $\delta^{(a_n)}_j =  h^{(a_n)}_j  \lambda^{(a_n)}_j$ .
\end{lemma}
\noindent The leading term in (\ref{deviation}) is, as expected by ergodicity, proportional to the ratio of lengths of the intervals, with 
% $\lambda^{(a_m)}_i / \lambda^{(a_n)}_j$. 
%$ \frac{  \lambda^{(a_m)}_i }{\lambda^{(a_n)}_j}$. 
the constant $\delta^{(a_n)}_j$ 
giving the density of elements of $\phi^{(a_n)}$ of type $j$ inside elements of  $\phi^{(a_m)}$. Let us remark that the leading term depends on the type $j$ only. The error, or the deviations from this leading behavior, decreseas exponentially in the number of visits to $\hat{K}$. 
% the number of matrices with positive contraction. % and it is strictly connected to deviations from Birkhoff averages. 
\begin{comment}
The purpose of the lemma is two-folded. First, deviations are controlled by a power $\gamma<1$ of the leading term and this form of the error will be essential in the proof of Proposition \ref{boundSf'growthprop}. Secondly, the speed of convergence is controlled uniformly in terms of the number $n-m$ of positive matrices with fixed contraction.
\end{comment}
\begin{proof} %Remark that $N_{A}^{(n,m)}_{ij}= A^{(n,m)}_{ij}$. 
Let us denote $\epsilon_{n} := (1-e^{-\overline{D}})^{n-1} \overline{D}$ where $\overline{D}>0$ is as in Property $(ii)$ in Lemma \ref{balancedaccelerationlemma}. 
%Choose $N>0$ so that $\epsilon_{N}<1/2$.
Let us prove first that for each $1\leq i,\, j\leq d $ and $m<n$ we have
\be\label{entrieslimit}
e^{-2\epsilon_{n-m} }  \lambda^{(a_m)}_i  \leq \frac{A^{( m,n) }_{i j}}{h^{(a_n)}_j} \leq e^{2 \epsilon_{n-m}} \lambda^{(a_m)}_i .
\ee
Consider the sets $\widetilde {A^{(m,l) }} \Delta_{d-1} \subset \Delta_{d-1}$, for $l > m$, which form a nested sequence of sets. Since by  (\ref{lengthsrelation}) we have  $\underline{\lambda}^{(a_m)} = A^{(m,l)}  \underline{\lambda}^{(a_l)}$, we also have 
%the vector $\underline{\lambda}^{(a_m)}$ renormalized by  $\lambda^{(a_m)}=\sum_i\lambda^{(a_m)}_i$ is such that
\bes
\frac{ \underline{\lambda}^{(a_m)}} { \lambda^{(a_m)} } \in \bigcap_{l > m} \widetilde {A^{(m,l)}} \Delta_{d-1} .
\ees
When $l= n$, since $D(A_i)\leq \overline{D}$ for each $i \in \mathbb{N}$ by in Property $(ii)$ in Lemma \ref{balancedaccelerationlemma}, applying $n-m-1$ times the contraction estimate (\ref{contraction}), we get
\be \label{contracteddiamater}
D( {A^{( m,n )} } ) \leq (1-e^{-\overline{D}})^{n-m-1} \overline{D} \leq \epsilon_{n-m} .
\ee
%where the last inequality follows by the choice (\ref{L_1def}) of $N$.
\noindent Denote by $\underline{e}_j$ the unit vector $(\underline{e}_j)_i=\delta_{ij}$ ($\delta$ is here the Kronecker symbol). Since both the vectors $\widetilde{A^{( m,n) } } \underline{e}_j $ and $ \frac{ \underline{\lambda}^{(m)}} { \lambda^{(m)} } $ belong to the closure of $\widetilde{A^{( m,n) } } { \Delta_{d-1} }$,
it follows by (\ref{contracteddiamater}), using compactness, that
\bes
d_H \left(  \frac{ \underline{\lambda}^{(a_m)}} {\lambda^{(a_m)} }  , 
\widetilde{A^{( m,n) } } \underline{e_j}  \right) = \log \frac{\max_{i=1,\dots , d} \frac{A^{ (m,n) }_{ij}}{{\lambda_i^{(a_m)}} }} {\min_{i=1 , \dots ,d}\frac{A^{ (m,n) }_{ij}}{{\lambda_i^{(a_m)}} }}    \leq \epsilon_{n-m} , 
\ees
where we also used the invariance of the distance expression by multiplication of the arguments by a scalar.
Equivalently, for each $1\leq i ,\, k \leq d$, 
\be
\label{projdistancecons}
e^{-\epsilon_{n-m}} \left( {A^{( m,n) }_{kj}}{{\lambda_i^{(a_m)}} }\right)
 \leq {A^{( m,n) }_{ij}}{{\lambda_k^{(a_m)}} } \leq e^{\epsilon_{n-m}}\left( {A^{( m,n) }_{kj}}{{\lambda_i^{(a_m)}} } \right)
\ee
and summing over $k$  or  respectively multiplying (\ref{projdistancecons})
 by $h^{(a_m)}_i$ and then sum  over both $k$ and $i$ and using that $\sum_i h^{(a_m)}_i \lambda ^{(a_m)}_i=1$ and (\ref{heightsrelation}), we get respectively
\be\label{1stratio}
e^{-\epsilon_{n-m}}\leq  \frac{{A^{( m,n) }_{ij}}{{\lambda^{(a_m)}} } }{{\sum_k A^{ (m,n) }_{kj}}{{\lambda_i^{(a_m)}} } }  \leq e^{\epsilon_{n-m}} , \qquad 
%\be
%\label{2ndratio}
e^{-\epsilon_{n-m}} \leq  \frac{ {h}^{(a_n)}_j {\lambda^{(a_m)}}}{ { \sum_k A^{ (m,n) }_{kj}}}  \leq e^{\epsilon_{n-m}} .
\ee
Producing the estimates in  (\ref{1stratio})  gives (\ref{entrieslimit}).
Since, for $n-m$ sufficiently large,  $\epsilon_{n-m} \leq 1/2$ and  $|1-e^{\pm 2 \epsilon_{n-m}}|\leq 4  \epsilon_{n-m}$, the Lemma follows from (\ref{entrieslimit}) by 
%Remark \ref{meaningentries}
remarking that $(N_{A})^{(m,n)}_{ij}= A^{(m,n)}_{ij}$  
 and setting $\delta_j^{(a_n)}:= {h}^{(a_n)}_j \lambda^{(a_n)}_j$.%, since $1-e^{2 \epsilon_N}\leq 2  \epsilon_N$.
\end{proof}
\subsubsection{Power form of the deviation}\label{powerformsec}
Let us show that, for times corresponding to a further appropriate acceleration of the cocycle $A$,  the deviations can be expressed as a small power  of the main order. %  for an appropriate further   
\begin{lemma}\label{powercontrolN}
%For all irreducible $\pi$ and  a.e. $\underline{\lambda} \in \Delta_d$,
For a.e.~$T$ there exists  a subsequence $\{b_k := a_{n_k}\}_{k\in \mathbb{N}} \subset \{a_n\}_{n\in \mathbb{N}}$  %we will relabel setting %$\{a_{n_k}\}_{k\in \mathbb{N}}$, where
% $b_k:= a_{n_k}$,  
%integer $K$ and 
and $0< \gamma < 1$ such that  for all $k \in \mathbb{N}$, for all $0\leq {k'}<k$, %for all $a_n  \leq a_{n_k}=b_k$,
  we have
\bes
(N_B)_{i j}^{({k'},k)}:= {(N_A)_{i j}^{(n_{k'} ,{n_{k}})}} = \delta_j^{( {b_k})} 
\left(  \frac{  \lambda^{(b_{k'})}_i }{ {\lambda_j^{({b_k})}} }  + {\E^{(b_{k'},b_{k})}_{i j}}  \right)
,  \qquad \left| \E^{(b_{k'},b_{k})}_{i j} \right| \leq  const { \left( \frac{  \lambda^{(b_{k'})}_i }{\lambda^{({b_k})}_j }\right)} ^{\gamma}   .
\ees
\end{lemma}
\begin{proof}
By Lemma \ref{expgrowthlemma}, there exists a measurable set $E_B \subset \hat{K}$ with positive measure and $\overline{C}_1>0$ such that if $\hat{T} \in E_B$ is recurrent to $E_B$ along  $\{a_{n_k}\}_{k\in \mathbb{N}}$ %is its sequence of  visits to $E_B$ 
(which is a subsequence of the visits to $\hat{K}$ since $E_B\subset \hat{K}$), (\ref{expgrowth}) holds. By ergodicity of $\hat{\Z}$, a.e.~$\hat{T}\in \hat{X}$ is recurrent to $E_B$. Thus for a.e.~$T$, there exists  $\hat{T} \in p^{-1}(T)$ recurrent to $E_B$ (indeed a full measure set of $\hat{T}$ in the fiber is recurrent) and   we can define $\{a_{n_k}\}_{k\in \mathbb{N}}$ to be the sequence along which $\hat{T}$ is recurrent.  
%\be\label{expgrowthA}
%\frac{ \ln \| A^{(n, n_k ) } (x) \|}{n_k - n} \leq \overline{C}_1, \qquad  \forall \, n < b_k .
%\ee
%Let us relabel the sequence by $b_k=a_{n_k}$. Hence, 

Since $(N_A)^{(n_{k'},{n_k})}_{i j}$
satisfies by Lemma \ref{ergodicconvergenceN} the estimate (\ref{deviation}) and $\delta_j^{(b_k)}\leq 1$, it is  
enough to prove that, for some\footnote{In the statement of Lemma \ref{powercontrolN} we require $0< \gamma <1$, but if (\ref{needederrorpowercontrol}) holds for some $\gamma\leq 0$, since $ \lambda^{(b_{k'})}_i / \lambda^{({b_k})}_j  \geq 1$ by positivity, it also holds for $0<\gamma'<1$.} $\gamma < 1$ and $const>0$, 
\be\label{needederrorpowercontrol}
 \left( \frac{  \lambda^{(b_{k'})}_i }{\lambda^{({b_k})}_j} \right)  \epsilon^{(b_{k'}, {b_k})}_{i j} \leq  const \left( \frac{  \lambda^{(b_{k'})}_i }{\lambda^{({b_k})}_j}\right) ^{\gamma}  .
\ee
By (\ref{normproduct}) and (\ref{balancedcomparisons}), 
%if $\overline{A}$ is the bounded type constant,  
$\lambda^{(b_{k'})}_i % \leq (\overline{A})^{k-k'} \lambda^{(n_k)} 
\leq d\nu \| {A}^{(n_{k'},n_k)}\| \lambda^{(b_k)}_j $.
Let $1-\gamma := - \log (1-e^{-D})/ \overline{C}_1 > 0$, so that $\gamma < 1$ % and let $\gamma = \min \{0,\gamma_1\}$.
 and, recalling the estimate of $ \epsilon^{(a_{n_{k'}},a_{n_k})}_{i j}= \epsilon^{(b_{k'},b_k)}_{i j}$ and using (\ref{expgrowth}), we have 
\bes
 { \left( \frac{  \lambda^{(b_k)}_j }{\lambda^{(b_{k'})}_i} \right)} ^{1-\gamma}\geq 
 {\left( \frac{1}{ d\nu \| {A}^{(n_{k'},n_k)}\|} \right)  } ^{1-\gamma} 
  \geq \frac{(d\nu)^{\gamma-1}}{ {\left( e^{\overline{C}_1({n_k-n_{k'}})}\right) }^{1-\gamma}}   =  (d\nu)^{\gamma-1} (1-e^{-D})^{n_k-n_{k'}}  \geq  %\frac{(1-e^{-D}) (d\nu)^{1-\gamma}}{4D} \left| 
  c \left| \epsilon^{(b_{k'},b_k)}_{i j} \right|
\ees
for some constant $c>0$, 
which is exactly (\ref{needederrorpowercontrol}).  
\end{proof}
%By Lemma \ref{expgrowthlemma}, there exists a measurable set $E' \subset K$ and $\overline{C}_1>0$ such that if  $\{b_k\}_{k\in %\mathbb{N}}$ is the subsequence of visits to $E'$, 
%%%\be\label{expgrowth}
%\frac{ \ln \| A^{(n, b_k ) } (x) \|}{b_k - n} \leq \overline{C}_1, \qquad  \forall \, n < b_k .
%\ee
%Remark that since $E'\subset K$, the sequence $\{b_k\}_{k\in \mathbb{N}}$ is a subsequence of  $\{a_n\}_{n\in \mathbb{N}}$, %i.e. its elements are of the form $b_k = a_{n_k}$ and hence 
\noindent We will denote  by $B$ the induced cocycle of $A$ corresponding to this acceleration. %, which is well defined for on a full measure set. 
In particular, for $k\in \mathbb{N}$,
 let   $B_k({T})= B_k(\hat{T})   = A ^{(b_k,b_{k+1})}(T)$ where $\{b_k\}_{k\in \mathbb{N}}$ is the sequence of visits of a chosen lift $\hat{T} \in p^{-1}(T)$ which is  recurrent to $E_B$.
%\bes
%B_k = A ^{(b_k,b_{k+1})}. 
%\ees

\subsubsection{Deviations for any interval.}\label{deviationstsec}
Let $T$,   $\{b_k\}_{k \in \mathbb{N}}$ and $0<\gamma<1$ be as in Lemma \ref{powercontrolN}. Let us now consider  an interval $I\subset I^{(0)}$ and let us denote % Let %$\{b_k\}_{k \in \mathbb{N}}$ be the sequence of Lemma \ref{powercontrolN}. and 
 the number of intervals of type $j$  of $\phi^{(b_k)}$ contained in $I$ by 
 %where $\{b_k\}_{k \in \mathbb{N}}$ is the sequence of Lemma \ref{powercontrolN}, i.e.~the quantities  
\bes
(N_B)^{(k)}_{j}(I) \doteqdot \# \{ \, l\, | \quad T^l I^{(b_k)}_{j} \subset I, \, 0\leq l < h^{(b_k)}_j \}.  
\ees
%by decomposing it into partition intervals. 
In order to describe the deviations from ergodic averages, set by convention $B_{-1}(T):= B_0(T)$ and introduce, for any  $0\leq k'<k$, the quantity % (and for $0<\gamma<1$ as in Lemma \ref{powercontrolN}), the function
\be\label{Thetadef}
\Th{k}{k'} = \Th{k}{k'}(T) : = \sum_{n=0}^{k-k'}\frac{\| B_{{k'}+n-1}(T) \|}{d^{\gamma n}}. 
\ee
The following lemma shows that the deviations of $(N_B)^{(k)}_{j}(I)$ can be estimated in terms of  $\Th{k}{k(I)}$, where
\be\label{kIdef}
\kei{I}:= \min \{k \st  \mathrm{there} \, \mathrm{exists} \,  F \in \phi^{(b_k)}\,  \mathrm{such }\,\, \mathrm{that}\, F \subset I \}. 
\ee
\begin{lemma}\label{deviationst}
%Let $\{b_k\}_{k \in \mathbb{N}}$ be the sequence of Lemma \ref{powercontrolN}. Then 
For a.e.~$T$  and for all $k\in \mathbb{N}$ and all $1\leq j \leq d$,  given any interval $I$ of length $\Leb{I}\geq \lambda^{(b_k)}_j$, we have   
%For a.e. $T \in E_1$, there exists a sequence $\{c_l\}_{l \in \mathbb{N}} \subset \{b_k\}_{k \in \mathbb{N}}$, i.e. of the form $c_l := b_{k_l}$ and an integer $K_1$ and $0<\gamma_1 < 1$, $const>0$ such that  for all $k_0 \in \mathbb{N}$, for all $t\geq \lambda^{(n_{{k_0}-K_1})}$, %for $s=0, 1$,
\bes
 (N_B)^{(k)}_{j }(I)  = \delta^{(b_k)}_j \left(  \frac{ \Leb{I}}{\lambda^{(b_{k})}_j} + \E^{(k) }_{j}(I)  \right),  \qquad \left|   \E^{(k) }_{j}(I)   \right|\leq const \,  \Th{k}{k(I)}(T)  \left(    \frac{ \Leb{I} }{\lambda^{(b_k)}_j}\right)^{\gamma} ,  
\ees
where %$\Leb{I}$ denotes the length of $I$ and
 $\{b_k\}_{k \in \mathbb{N}}$, $0<\gamma<1$ and $\delta^{(b_k)}_j$ are the same than in Lemma \ref{powercontrolN}.  
\end{lemma}
\begin{proof} 
Let us decompose $I$ into elements of the partitions  $\phi^{(b_k)}$, $k\geq k(I)$, as follows. 
Consider all intervals of  $\phi^{(b_{\kei{I}})}$ which are completely contained in $I$. By definition of $\kei{I}$, this set is not empty and moreover, $I$ is contained in at most two intervals of  $\phi^{(b_{\kei{I}-1})}$. 
Hence, if we denote  by $\#^{\kei{I}}_i$ the number of intervals of $\phi^{(b_\kei{I})}$ of type $i$ contained in $I$, % by (\ref{entriesmeaningdecomposition})
 if $k(I)>0$, we have  $\#^{\kei{I}}_i \leq 2 \max_{1\leq l \leq d} (B_{\kei{I}-1})_{ l i}
 \leq 2 \| B_{\kei{I}-1}\|$. If $k(I)=0$, since $I^{(0)}$ contains $\sum_{1\leq l \leq d}(B_0)_{l i}$ elements of $\phi^{(b_0)}$ of type $i$, we have $\#^{0}_i \leq  \norm {B_0}  $.  Thus, if by convention we set $B_{-1}:=B_0$, we have 
\bes
\Leb{I} = \sum_{i=1}^{d} \#^{\kei{I}}_i   \lambda_{i}^{(b_{\kei{I}})} + \delta (I, \kei{I}), \qquad  \#^{\kei{I}}_i \leq 2 \|  B_{\kei{I}-1}\| , \quad \delta(I, \kei{I}) \leq 2\max_{1\leq i\leq d}  \lambda_{i}^{(b_{\kei{I}})}, 
\ees
where $\delta (I, \kei{I})$ is the length of the reminder (possibly empty), given by the two intervals (at the two ends) left after subtracting from $I$ all interval of $\phi^{(b_{\kei{I}})}$ completely contained in it. Decompose in the same way the two reminders into intervals of $\phi^{(b_{\kei{I}+1})}$ (if any) completely contained in it and two new reminder intervals and so on by induction, until decomposing into elements of $\phi^{(b_k)}$. As before, if  $\#^{k'}_i$ is the number of intervals of $\phi^{(b_{k'})}$ of type $i$ involved in the decomposition,  $\#^{k'}_i \leq 2 \| B_{k'-1}\| $ since by construction each of the two reminders is contained in an interval of  $\phi^{(b_{k'-1})}$. Thus, we get %(with the same convention that $B_{-1}:=B_0$) we get
\be \label{decompt}
\Leb{I} =  \sum_{k'=\kei{I}}^{k} \sum_{i=1}^{d} \#^{k'}_i   \lambda_{i}^{(b_{k'})} + \delta (I, k), \qquad  \#^{k'}_i \leq 2 \| B_{k'-1} \|, \quad \delta(I, k) \leq 2 \max_{1\leq i\leq d}  \lambda_{i}^{(b_{{k}})} .
\ee  
Using this decomposition to estimate $(N_B)^{(k) }_{j }(I)$ in terms of $(N_B)^{(k' ,k)}_{i j}$ and then applying Lemma \ref{powercontrolN},  we get 
% and Corollary \ref{powercontrolN}, 
\bes %\label{NtusingNij}
{(N_B)}^{(k) }_{j }(I) =  \sum_{k'=\kei{I}}^{k} \sum_{i=1}^{d} \#^{k'}_i   (N_B)^{(k' ,k)}_{i j}  = \delta^{(b_{k})}_i  \sum_{k'=\kei{I}}^{k} \sum_{i=1}^{d} \#^{k'}_i \left(   \frac{  \lambda^{(b_{{k'}})}_i }{\lambda^{(b_k)}_j} + \E_{i j}^{(b_{k'},b_{k})} \right) .%+ N_{\delta(I, k)}, \qquad  N_{\delta(I, k)} \leq  \frac{\max_l\lambda_{l}^{(n_{k})}}{\lambda^{(n_{k})}_{j} }
\ees
%where $ N_{\delta(I, k))}$ is the cardinality of intervals of $\phi$ contained in the very last reminder intervals. %of order $k$ and hence satisfies the estimate in (\ref{NtusingNij}).
%Applying Lemma \ref{powercontrolN} to estimate each $ (N_B)^{(k' ,k)}_{i j} $, we get
%\be \label{NtusingNij}
% (N_B)^{(k)}_{j }(I) = \delta^{(b_{k})}_j  \sum_{k'=\kei{I}}^{k}} \sum_{j=1}^{d} \#^{k'}_j \left(   \frac{  \lambda^{(b_{{k'}})}_i }{\lambda^{(b_k)}_j} + E_{i j}^{(b_{k'},b_{k}})} \right) .%   + N_{\delta(I, k)} ,
%\ee
%\delta^{(n_k)}_j \left(  \frac{  \lambda^{(n_{k'})}_i }{\lambda^{(n_k)}_j} + E^{(n_{k'},n_k)}_{i j} \right),  \qquad \left| E^{(n_{k'},n_k)}_{i j} \right|\leq  const \left( \frac{  \lambda^{(n_{k'})}_i }{\lambda^{(n_k)}_j}\right) ^{\gamma}   .
Recalling (\ref{decompt}), we have
\be \label{3errors}
 (N_B)^{(k)}_{j }(I) = \delta^{(b_{k})}_j  \frac{\Leb{I}}{\lambda^{(b_k)}_j }  + \delta^{(b_{k})}_j \sum_{k'=\kei{I}}^{k} \sum_{i=1}^{d} \#^{k'}_i   \E_{i j}^{(b_{k'},b_{k})} -  \delta^{(n_{k})}_j \frac{ \delta(I, k)  }{\lambda^{(b_{k})}_{j} }   .
% \qquad  N-{\delta(t, {k_0}\!-\!K)} \leq  \frac{\max_l\lambda_{l}^{(n_{{k_0}-K})}}{\lambda^{(n_{k_0})}_{j} }
\ee 
\noindent  In oder to conclude, let us show that each of the two last terms contributes to an error of the desired form. The very last term in (\ref{3errors}) is less then $2\nu$ by (\ref{decompt})  and balance (\ref{balancedcomparisons}) and hence, since  $\Leb{I}/ \lambda^{(b_k)}_j \geq 1$  by assumption, it is controlled by choosing the constant appropriately.  
%To estimate the first error, since $T$ is bounded type, $a^{k}_j \leq \overline{A}$ for each $k_t\leq k \leq {k_0}-K$. Hence, by 
For the other term, applying Lemma \ref{powercontrolN}, 
\bes
\sum_{k'=\kei{I}}^{k} \sum_{i=1}^{d} \#^{k'}_i  \E_{i j}^{(b_{k'},b_k)} \leq
% const  \sum_{k'=\kei{I}}^{k}} \|  B_{k'-1}\| { \left( \frac{  \lambda^{(b_{k'})}_i }{\lambda^{({b_k})}_j }\right)} ^{\gamma}= 
 const  \left( \frac{ \Leb{I} }{\lambda^{(b_{k})}_j}\right) ^{\gamma}   \sum_{k'=\kei{I}}^{k} \|  B_{k'-1}\|  \left( \frac{  \lambda^{(b_{k'})}_i }{\Leb{I}}\right) ^{\gamma}.
\ees 
Since by definition of $\kei{I}$, (\ref{positvegrowth}) and balance (\ref{balancedcomparisons}) we have $\Leb{I} \geq \frac{1}{d\nu} \lambda^{(b_{\kei{I}})} \geq \frac{1}{d \nu} d^{k'-\kei{I}} \lambda^{(b_{k'} )}_i $, the sum is controlled as desired by
\bes
  \sum_{k'=\kei{I}}^{k} \|  B_{k'-1}\| \left( \frac{  \lambda^{(b_{k'})}_i }{\Leb{I}}\right) ^{\gamma} \leq {(d\nu)^{\gamma}}  \sum_{k'=\kei{I}}^{k} \frac{  \|  B_{k'-1}\| }{  d^{\gamma(k'-\kei{I})}}  =
   {(d\nu) ^{\gamma}} \, 
% \sum_{j=0}^{k-\kei{I}}} \frac{  \|  B_{\kei{I} +j-1}\| }{  d^{\gamma(j)}} 
   \Th{k}{\kei{I}}.
\ees 
%Since $\delta(t, {k_0}\!-\!K) \leq 1$ and $t\geq \lambda^{(n_{k_0})}_j$, also the second term in (\ref{3errors}) is of desired form if $\gamma_1=\max \{ 0, \gamma\}$. 
\end{proof}
In \S\ref{devfromarithsec}) we consider more in general intervals $I=(a,b) \subset \mathbb{R}$, of length $\Leb{I}\leq 1$ and either $a$ or $b$ belonging to the singularities set $\{\overline{z}^+_0, \dots, \overline{z}^+_{s_1}, \overline{z}^-_0, \dots, \overline{z}^-_{s_2}\}$. We consider $I$  as a subset of $I^{(0)}$ \emph{modulo one}, i.e.~ if $a<1<b$, we consider the union $(a,1) \cup (0,b-1)$ and if $a<0<b$, we consider $(0,b) \cup (a+1,1)$. 
One can decompose also this type of intervals so that (\ref{decompt}) holds. Indeed, if $I$ modulo one is a disjoint union, then $k(I)=0$ since one of the two intervals ($(a,1)$ or $(0,b)$ respectively) is union of elements of $\phi^{(b_0)}$, whose total number is bounded by $\|  B_{0}\|$, and the other interval can be decomposed as before. Thus, the same proof that shows Lemma \ref{deviationst} gives also the following remark.
\begin{rem}\label{deviationsforImod1}
Lemma \ref{deviationst} holds also for intervals $I$ modulo one such that $\Leb{I}\leq 1$ and one of the endpoints of $I$ is a singularity of $f$.  
\end{rem}

\subsection{Cancellations.}\label{cancelletionssec}
%Let us give the proof of Proposition \ref{boundSf'growthprop}
Let $b_{k}$ be a an induction time of the sequence $\{b_k\}_{k\in\mathbb{N}}$ of  Lemma \ref{deviationst}. Let $x_0\in I^{(b_k)}_{j_0}$ and let $r_k:=h^{(b_k)}_{j_0}$.  Consider the  distances of the points in the orbit segment $\{T^i x_0\}_{i=0}^{r_k-1}$ from the right singularity $\overline{z}^+_i$,  for $i=0, \dots, s_1-1$,  taken modulo one\footnote{For example, if $T^jz_0< \overline{z}^+_i$, then $T^j z_0 - \overline{z}^+_i \mod 1 = 1+ T^j z_0 - \overline{z}^+_i $. In this way, since $u_i(x)$ and $v_i(s)$ %= u(\{x-x_i\})$
 are $1$-periodic, the quantity $T^j z_0 - \overline{z}^+_i \mod 1 $ (respectively $ \overline{z}^-_i  - T^j z_0 \mod 1$) gives the value of $1/u_i(T^jz_0)$  (respectively $1/v_i(T^jz_0)$).} (i.e. $ T^j z_0 - \overline{z}^+_i \mod 1,  0\leq j < r_k$), and, respectively, consider the distances from the  left singularity $\overline{z}^-_i$, for $i=0, \dots, s_2-1$, taken modulo one (i.e. $ \overline{z}^-_i  - T^j z_0 \mod 1,  0\leq j < r_k$). 
%\bes
% \{ T^j z_0 - \overline{z}_i \}, \quad  0\leq j < r_k-1; \qquad  \{ \overline{z}_i  - T^j z_0 \}, \quad  0\leq j < r_k-1,  
%\ees 
Rearrange each group in increasing order, renaming  by $x_i{(j)}$ (or respectively $y_i{(j)}$) the $j^{th}$ distance from the right (respectively from the left), so that the following equalities of sets hold\footnote{Here the notation $\{x\}$ denotes the singleton set containing $x \in \mathbb{R}$ as its element and should not be confused with the fractional part, that we denote by $\fracpart{x}$.}  
\begin{eqnarray}
&& \bigcup_{j= 0}^{ r_k-1} \left\{  x_i{(j)} \right\}  =   \bigcup_{j= 0}^{ r_k-1} \left\{  T^j z_0  - \overline{z}^+_i  \mod 1 \right\}, \quad 
 x_i{(j_1)}< x_i{(j_2)} \, \, \,  \forall j_1< j_2  \quad (0 \leq i <  s_1); \label{rearrangingx} \\
&&  \bigcup_{j= 0}^{ r_k-1} \left\{   y_i{(j)} \right\} =   \bigcup_{j= 0}^{ r_k-1}   \left\{  \overline{z}_i^- - T^j z_0  \mod 1 \right\}, \quad 
 y_i{(j_1)}< y_i{(j_2)} \, \, \, \forall j_1< j_2 \quad (0 \leq i < s_2). \label{rearrangingy} 
\end{eqnarray}
%\begin{eqnarray}
%\label{rearrangingx}
%&& \left\{  x_i{(j)}, \, j= 0, \dots , r_k-1 \right\} =   \left\{ T^j z_0 - \overline{z}_i  \mod 1, \, j= 0, \dots , r_k-1 %\right\}, \quad
% x_i{(j_1)}< x_i{(j_2)}, \, \forall j_1< j_2 ; \\
%&& \left\{  y_i{(j)}, \,  j= 0, \dots , r_k-1 \right\} =   \left\{ T^j z_0 - \overline{z}_i \mod 1, \, j= 0, \dots , r_k-1 %\right\}, \quad 
% y_i{(j_1)}< y_i{(j_2)}, \,  \forall j_1< j_2 ;
%\label{rearrangingy}
%\ee
%points of  $\{T^i x_0\}_{i=0}^{r-1}$ in , so that 
%\bes 0<x_0<x_1< \dots < x_{r-1}<1, \qquad \bigcup_{i=0}^{r-1} \{ T^i x_0 \} = \bigcup_{i=0}^{r-1} \{ x_i \}.
%\ees
%%$0<x_0<x_1< \dots < x_{r-1}<1$ and  $ \bigcup_{i=0}^{r-1} \{ T^i x_0 \} = \bigcup_{i=0}^{r-1} \{ x_i \}$.
%Similarly rearrange in increasing order distances from $1$, i.e. the elements of $\{1-T^i x_0\}_{i=0}^{r-1}$, renaming them by 
%%$0<y_0 < y_1 < \dots < y_{r-1}<1$,   $ \bigcup_{i=0}^{r-1} \{ 1- T^i x_0 \} = \bigcup_{i=0}^{r-1} \{ y_i \}$.
%\bes 0<y_0 < y_1 < \dots < y_{r-1}<1,   \qquad \bigcup_{i=0}^{r-1} \{ 1- T^i x_0 \} = \bigcup_{i=0}^{r-1} \{ y_i \}.
%\ees

\subsubsection{Deviations from an arithmetic progression}\label{devfromarithsec}
As a consequence of Lemma \ref{deviationst}, we have the following.
For $j=0, \dots, r_k-1$, let $I_i^{+}(j)$,  for $i=0,\dots, s_1-1$, be the interval $(\overline{z}^+_i, \overline{z}^+_i + x_i(j))$ considered modulo one and let  $I_i^{-}(j) $,  for $i=0,\dots, s_2-1$,  be the interval $( \overline{z}^-_i - y_i(j),\overline{z}^-_i)$ considered modulo one.  For brevity, let $k_i^{\pm}(j):=k\left({I^{\pm}_i(j)}\right)$ (see the definition given in (\ref{kIdef})).

\begin{cor}\label{deviationsxjyj}
%There exist $C>0$ and $\gamma<1$ such that, There exists $C^k_{j_0}>0$ such that
For all $i=0, \dots , s-1$, $1\leq j < r_k$, we have
%\be\label{deviationsxj}\label{deviationsyj} 
%x_i(j) = C^k_{j_0} \lambda^{(n_{b_k})}_{j_0} \left( j + O\left(  \Th{k}{k^{+}_i(j)}(T) \, j^{\gamma}\right) \right) , \qquad 
%y_i(j) = C^k_{j_0} \lambda^{(n_{b_k})}_{j_0} \left( j + O\left(  \Th{k}{k^{-}_i(j)}(T) \, j^{\gamma}\right) \right) .
%\ee
%\end{cor}
\begin{eqnarray}
&& x_i(j) =\frac{ \lambda^{({b_k})}_{j_0}}{\delta^{(b_k)}_{j_0}} \left( j + O\left(  \Th{k}{k^{+}_i(j)} \, \left( \| B_{k}\| \, j \right)^{\gamma }\right) \right) 
\label{deviationsxj}\\
&& y_i(j) = \frac{ \lambda^{({b_k})}_{j_0}}{\delta^{(b_k)}_{j_0}} %\lambda^{(n_{b_k})}_{j_0} 
\left( j + O\left(  \Th{k}{k^{-}_i(j)} \,  \left( \| B_{k}\| \, j \right)^{\gamma }\right) \right) 
.\label{deviationsyj}
\end{eqnarray}
\end{cor}
%Assume that  $T$ be bounded type and $n=n_k$ is an induction time of the sequence $\{n_k\}_{k\in\mathbb{N}}$ of balanced times given by Lemma \ref{deviationst}. 
%\begin{lemma}
%There exist $C>0$ and $\gamma<1$ such that
%\be
%x_j = C \lambda^{(n)}_{j_0} \left( j + o(j^{\gamma}) \right) \qquad y_j = C \lambda^{(n)}_j \left( j + o(j^{\gamma}) \right) 
%\ee
%for all $j=1, \dots, r-1$.
%\end{lemma}
%\begin{rem}\label{minimumdistance}
%Remark that each of the points of $\{T^i x_0\}_{i=0}^{r-1}$ belong to a different floors of the tower  $Z^{(n)}_{j_0}$, hence, in particular, the minimum distance $\min_{i\neq j} |x_i -x_j|$ is bounded below by $\lambda^{(n)}_{j_0}$.
%\end{rem}
\begin{proof} 
Consider the interval $I^{+}_i(j)$.  Since by definition $x_i(j)$ is the distance  of the $j^{th}$ closest point to  the right of $\overline{z}^+_i$, there are $j$ points of the orbit in $I^{+}_i(j)$, so $ (N_B)^{(k)}_{j_0 }(I^{+}_i(j))  = j $. Hence, Lemma  \ref{deviationst} together with Remark \ref{deviationsforImod1} gives
\bes
 j  = \delta^{(b_k)}_{j_0} \left(  \frac{x_i(j)  }{\lambda^{(b_{k})}_{j_0}} + \E^{(k) }_{j_0}(I^{+}_i(j))  \right),  \qquad \left|   \E^{(k) }_{j_0}(I^{+}_i(j))  \right|= O\left(  \Th{k}{k^+_i(j)}\,  \left( \frac{ x_i(j) }{\lambda^{(b_{k})}_{j_0}}\right)^{\gamma} \right) ,
\ees
for some $0< \gamma<1$, or, rearranging the terms, 
\bes
 x_i(j)  =  \lambda^{(b_{k})}_{j_0} \left( \frac{j}{\delta^{(b_{{{k}}})}_{j_0}}    -  \E^{({k}) }_{j_0 }(I^{+}_i(j)) \right) .
%,  \qquad \left|   E^{(n_{k_0}) }_{j ,0}(t)  \right|\leq const  \left( \frac{ x_i(j) }{\lambda^{(n_{k_0})}_i(j)}\right)^{\gamma_1}  ,
\ees
This gives (\ref{deviationsxj}) %with $C^k_{j_0} := (c^{(b_{{{k}}})}_{j_0})^{-1}$
 if we show that $ x_i(j) / \lambda^{(b_{k})}_{j_0}\leq c\|B_k \| j$ for some constant $c$. Since by definition  $\{b_k\}_{k\in \mathbb{N}}$ is a subsequence of a balanced sequence, we have $B_k>0$. Thus, inside each element of $\phi^{(b_{k-1})}$ there is at least one element of $\phi^{(b_{k})}$ of type $j_0$, or, equivalently, one point of the orbit $\{T^i x_0\}_{i=0}^{r_k-1}$. Since a lower bound for the number of elements of  $\phi^{(b_{k-1})}$ in $I^+_i(j)$ is given by $ [ x_i(j) / \lambda^{(b_{k-1})}]$, where $[ \cdot ]$ denotes the integer part, we get  $j\geq x_i(j) / \lambda^{(b_{k-1})}-1$. Using that $\lambda^{(b_{k-1})} \le\ \| B_k\| \lambda^{(b_{k})}$, we get $ x_i(j) \leq \| B_k\|  (j+1)\lambda^{(b_{k})} $, which concludes the proof. 

The proof of (\ref{deviationsyj}) follows in the same way applying Lemma  \ref{deviationst} together with Remark \ref{deviationsforImod1} to the interval  $I^{-}_i(j)$.
% from estimate of $(N_B)^{(k)}_{j_0 }(I^{-}_i(j))  $ in Lemma  \ref{deviationst} in the exact same way.
\end{proof}

%\begin{rem}
\label{Scheglov}
Let us remark that, in the special case in which the permutation $\pi$ is $(54321)$ and  $\overline{z}_i^{\pm}$ are the endpoints of a subinterval $I^{(0)}_i$ of $T$ (see the Footnote \ref{Scheglovsymm}, page \pageref{Scheglovsymm}), Scheglov in \cite{Sch:abs} shows that  for a.e.~$\underline{\lambda}$ one can find a subsequence of times $\{b_k\}_{k\in \mathbb{N}}$ and a constant $K>0$ such that the $|x_i(j)-y_i(j)| \leq K \lambda^{(b_{k})}_{j_0} $.  
%(\ref{deviationsxj}) and (\ref{deviationsyj})
 %is a bounded error $O(1)$.
  This stronger form of control of the deviations, which presumably holds for all combinatorics of the form  $(n \,  n\!-\!1 \dots 21)$, exploits crucially the symmetry of the permutation and hence can  be used to prove Theorem \ref{multivalHthm} only for $g=2$ (see Footnote \ref{n21}, page \pageref{n21}). % and isometric saddles.  
%\end{rem}

\subsubsection{Acceleration for cancellations}
Let us now accelerate one more time in order to prove Proposition \ref{boundSf'growthprop}.
%the sequence $\{n_k\}_{k\in \mathbb{N}}$ satisfy (\ref{boundSf'growth}) of 
\begin{proofof}{Proposition}{boundSf'growthprop}
By Lemma \ref{subexpgrowthlemma} applied to\footnote{One can set $\epsilon= \overline{\gamma} \ln ( d / t ) $ where $t$ is any number $1<t<2$, so that $d\geq 2 > t$ gives $\epsilon>0$. Later we need $t>1$. Here, for concreteness, we chose $t=3/2$.} $\epsilon= \overline{\gamma} \ln ( 2 d / 3) \geq 0 $ where $\overline{\gamma}:=\min\{\gamma, 1-\gamma\} \geq 0 $, we can find $E_C\subset E_B$ such that 
%for a.e.~$T$, 
if $\hat{T} \in p^{-1}(T)$ is recurrent to $E_C$ along  $\{b_{k_l}\}_{l\in \mathbb{N}}$  (which is a subsequence of the return times $\{b_k\}_{k\in \mathbb{N}}$ to $E_B$), we have
% $\{c_l\}_{l\in \mathbb{N}}$, which are a subsequence of the return times $\{b_k\}_{k\in \mathbb{N}}$ to $E_B$ and hence of the form $c_l:=b_{k_l}$, then we have
\be\label{subexpgrowtheql}
 \| B_{k_{l} - m  }  \|  \leq \overline{C}_2 ( 2 d/3)^{\overline{\gamma} m} , \qquad \forall \, 0\leq  m \leq k_l. 
\ee
We remark that a.e.~$T$ has a lift $\hat{T}$ recurrent to $E_C$, by ergodicity of $\hat{\Z}$. Let us set $c_l:=b_{k_l}$  
%Let $\{n_k\}_{k\in \mathbb{N}}$ be the sequence given by Lemma \ref{deviationst}.
%Let  $r_{k_0} \in \{n_k\}_{k\in \mathbb{N}}$, 
and let $x_0 \in I^{(c_{l})}_{j_0}$, $r_{l} =h^{(c_{l})}_{j_0}$.
Recall that that $f'$ is of the form (\ref{logsing}). Since $g$, being bounded variation, is in particular bounded by some $M_g>0$, we have $\BS{g}{r_l}\leq M_g r_l$. Thus, we can assume without loss of generality that  $f'=  \sum_{i=0}^{s-1}\left(  C_i^- v_i - C_i^+ u_i \right)$. 

For clarity, let us first give the proof in the special case $s_2=s_1=1$ and $f'(x)=- C^+_0/x + C^-_0/(1-x) $ with $C_0^+=C_0^-=C$,
 since the notation is heavier in the general case. 
Using the relabeling (\ref{rearrangingx}, \ref{rearrangingy}) for $s=1$, let $x_0(j)$ and $y_0(j)$ denote the distances of the $j^{th}$ orbit point from $0$ and $1$ respectively, so that 
\bes \BS{f'}{r_{l}} (z_0) =   \sum_{j=0}^{r_l-1} \left( \frac{C}{1- T^j z_0} - \frac{C}{T^j z_0}\right)% C \sum_{j=0}^{r_l-1} \left({v_i}(T^j z_0) - {u_i}(T^j z_0)\right)  = 
 = \sum_{j=0}^{r_l-1} \left( \frac{C}{y_0(j)} - \frac{C}{x_0(j)}\right) =  C \sum_{j=0}^{r_l-1}  \frac{x_0(j)- y_0(j)}{ x_0(j) y_0(j)}  .
\ees
\noindent We remark that each of the points in $\{T^i z_0\}_{i=0}^{r_l-1}$ belongs to a different floor of the tower  $Z^{(c_l)}_{j_0}$, hence,   $\min_{i\neq j} |T^{i} z_0 -T^{j} z_0 | \geq \lambda^{(c_l)}_{j_0}$ and we can estimate the denominator using that $x_0(j) , y_0(j) \geq j \lambda^{(b_{k_l})}_{j_0}$ when $j\geq 1$. 
To estimate the numerator, let us apply Corollary \ref{deviationsxjyj}. By (\ref{deviationsxj}) and (\ref{deviationsyj}) the leading term in $j$ in  $x_0(j)$ and $y_0(j)$ for $j \geq 1$ are the same  and cancel out. Moreover, $1/{\delta^{(c_l)}_{j_0}}\leq d \nu^2$ by balance (\ref{balance}). So, setting aside the contribution of the two closest points $x_0(0)= x^{min}_0$ and $y_0(0)= y^{min}_0$, we get
\begin{comment}
x_i(j) =\frac{ \lambda^{(c_l)}_{j_0}}{\delta^{(k_l)}_{j_0}} \left( j +  \right) 
\label{deviationsxj}\\
&& y_i(j) = \frac{ \lambda^{(c_l)}_{j_0}}{\delta^{(k_l)}_{j_0}} %\lambda^{(n_{b_k})}_{j_0} 
\left( j + O\left(  \Th{k}{k^{-}_i(j)} \,  \left( \| B_{k_l}\| \, j \right)^{\gamma }\right) \right) 
.\label{deviationsyj}
\end{comment}
\bes
\left|  \BS{f'}{r_{l}}  -\frac{C}{y^{min}_0} +  \frac{C}{x^{min}_0} \right| \leq  \sum_{j=1}^{r_l-1} \frac{  j^{\gamma}  \lambda^{(b_{k_l})}_{j_0} \|B_{k_l}\|^{\gamma} 
\left( O\left(  \Th{k_l}{k^{+}_0(j)}  \right)  + O\left(  \Th{k_l}{k^{-}_0(j)}  \right) \right)   }{ j^2 (\lambda^{(b_{k_l})}_{j_0})^2 } . %\right| \leq  \sum_{i=J}^{r_k-1}  \frac{\left| y_i -x_i \right| }{x_i y_i }  \leq  \sum_{i=J}^{r_k-1} \frac{  \lambda^{(n_{k_0})}_{j_0} o( j^{\gamma})}{ {j^2  \left( \lambda \right)^2 } }\leq C r_k ,
 \ees
Let us first bound the part of the above sum which involves $\Th{k_l}{k^{-}_0}(j)$. Let $I^k_{1^-}$ be the interval of $\phi^{(b_k)}$ which contains $1$ as a right endpoint. Set by convention $I^{(-1)}_{1^-}:= I^{(0)}$. The intervals $\{I^k_{1^-}\}_{k\in \mathbb{N}}$ are nested and we can use them to rearrange the sum over $j$ as follows. Let us remark that if $y_0(j) \in %\widetilde{I^{k'}_{1-}}  :=
 I^{k'-1}_{1^-} %(0)
 \backslash I^{k'}_{1^-}$, then,  by definition of $k^-_0(j)$, we have $k^{-}_0(j)=k'$. Moreover, since we kept aside the closest point to $1$, there are no orbit points in $I^{k_l}_{1^-}$. Thus, recalling also that $ \| B_{k_l}\|\leq \overline{C}_2 $ by (\ref{subexpgrowtheql}),  we obtain 
\bes
\sum_{j=1}^{r_l-1} \frac{ % j^{\gamma}  \lambda^{(b_{k_l})}_{j_0}
 \|B_{k_l}\|^{\gamma} 
  O\left( \Th{k_l}{k^{-}_0(j)}  \right)  }{ j^{2-\gamma} \lambda^{(b_{k_l})}_{j_0} } \leq
const\,  \sum_{k'=0}^{k_l}  \sum_{  y_0(j) \in  I^{k'-1}_{1^-} %(0)
 \backslash I^{k'}_{1^-}}
%\in I_{j_{k'}^-}^{(b_{k'})} \backslash I_{j_{k'+1}^-}^{(b_{k'+1}) }} 
\frac{   %\|B_{k_l}\|^{\gamma} 
  \Th{k_l}{k'}  %k^{-}_0(j)}
  }{ j^{2-\gamma} \, \lambda^{(b_{k_l})}_{j_0}  }\leq   const\, r_l \sum_{k'=0}^{k_l}  
\frac{  \Th{k_l}{k'}}  {d^{(1-\gamma)(k_l-k')}}, 
 \ees
where in the last inequality we used that by balance  $ \lambda^{(b_{k_l})}_{j_0}  \geq const (r_l)^{-1}$ and that, if $j_{k'}$ denotes the minimum $j$ such that $ y_0(j) \in I^{k'-1}_{1^-}  \backslash I^{k'}_{1^-}$,  using balance and positivity as in the proof of Corollary \ref{deviationsxjyj}, we get 
\be\label{jestimate}
\sum_{ y_0(j) \in  I^{k'-1}_{1^-} %(0)
 \backslash I^{k'}_{1^-}}  \frac{1}{j^{2-\gamma}} = O \left(\frac{1}{j_{k'}^{1-\gamma}}\right) \quad \mathrm{and} \quad 
 j_{k'}\geq \left[ \frac{\min_j \lambda^{(b_{k'})}_j}{ \lambda^{(b_{k_l}-1)}} \right] \geq const \,  d^{k_l-k'}. 
% \leq const \,  \frac{1}{d^{k_l-k'}}, 
\ee
   %$j_k\geq \min_j \lambda^{b_{k'}}_j/ \lambda^{(b_{k_l}-1} \geq const d^{k_l-k}$. since 
Recalling that $\overline{\gamma}=\min\{\gamma, 1-\gamma\}$ and the definition (\ref{Thetadef}) of $ \Th{k_l}{k'}$,   and  changing indexes by $k=k_l-k'$ first and $m=k-n$ later in order to rearrange the sums, since $k_l$ by (\ref{subexpgrowtheql}) is such that $\| B_{k_l - m-1 } \| \leq  const {\left(\frac{2d}{3} \right)}^{\overline{\gamma} m}$,   we have
\bes
\begin{split}
 \sum_{k'=0}^{k_l}   \frac{  \Th{k_l}{k'}}{d^{(1-\gamma)(k_l-k')}} & \leq  \sum_{k=0}^{k_l} \frac{
%\Th{k_l}{k'}  = \sum_{j=0}^{k-k'}\frac{\| B_{{k'}+j-1}(T) \|}{d^{\gamma j}}.
  \Th{k_l}{k_l-k'}}{d^{(\overline{\gamma})k}} \leq \sum_{k=0}^{k_l}\sum_{n=0}^k \frac{\| B_{k_l-k+n-1}\|}{d^{\overline{\gamma}(n+k)}} =  \sum_{k=0}^{k_l}\sum_{m=0}^k \frac{\| B_{k_l-m-1}\|}{d^{\overline{\gamma}(2k-m)}} \\ &
  \leq \sum_{k=0}^{k_l} \frac{ const }{d^{\overline{\gamma}(2k) }}  \sum_{m=0}^k 
%  \left( \frac{ d^{2 \overline{\gamma}  }}{2} \right) ^m
 \left( \frac{ 2 ^ {\overline{\gamma}} (d^2)^{ \overline{\gamma}  }}{ 3^ {\overline{\gamma}} } \right) ^m, 
\end{split}
\ees
which, since $(2d^2/3)^ {\overline{\gamma}}>1$ and thus $\sum_{m=0}^k ( 2d^{2}/3)^{ \overline{\gamma} m} = O(d^{2\overline{\gamma}k} (2/3)^k )$,  is bounded by $const$ $\sum_{k=0}^{k_l} (2/3)^{ \overline{\gamma} k}$, which is uniformly bounded independently on $l$.

The proof that also the sum involving $\Th{k_l}{k^{+}_0}(j)$ is uniformly bounded in $l$ is analogous and gives also an upper bound by a fixed constant. This concludes the proof in this case.

\begin{comment}Consider first the points $x_i(j) \in [0, \lambda^{n_{k_0 - K_1}}] $, where $K_1$ is given by Lemma  \ref{deviationst}. 
 Let $J_x=\min \{ j | \, x_i(j) \notin [0, \lambda^{n_{k_0 - K_1}}[ \} $ and $J_y = \min \{ j | \,  y_i(j) \notin [1- \lambda^{n_{k_0 - K_1}}, 1]\}$. Let $J=\max \{ J_x, J_y \}$. Assume first  that $J=J_x $.
  
Since $T$ is bounded type and the points have distance at least $\lambda^{(n)}_{j_0}$, then $J_x \leq \lambda^{n_{(k_0 - K_1})}/\lambda^{(n_{k_0})}_{j_0} \leq \overline{A}^{K_1}$. Moreover, for each $j\geq 1$, $x_i(j)\geq x_1\geq \lambda^{(n_{k_0})}_{j_0}$, so that, using balance (\ref{balancedcomparison}),  $1/x_i \leq d \nu^2 h^{(n_{k_0})}_{j_0}$. Hence, considering separately the contribution $1/x_0$ of $x_0$, which could be arbitrarily close to the origin, the contribution of $x_i(j) \in [0, \lambda^{n_{k_0 - K_1}}] $ for $j\geq 1$ is bounded by  $\overline{A}^{K_1}d \nu^2 r^{(n_{k_0})} $. Reasoning in the same exact way, also the contribution of  $y_i(j)$ $j=1,\dots, J$ is bounded by $1/y_0 + \overline{A}^{K_1}d \nu^2 r^{(n_{k_0})}$ . If $J=J_y $, the same argument can be repeated to get  $J=J_y  \leq  \overline{A}^{K_1}$.
Hence
\bes
\leq  \frac{1}{x_0} + \frac{1}{y_0} + 2 \overline{A}^{K_1}d \nu^2 r_l .
\ees
Now consider the other points. By definition of $J$, for $j\geq J$, $x_i(j) \notin  [0, \lambda^{n_{k_0 - K_1}}[ $ and $y_i(j) \notin [1- \lambda^{n_{k_0 - K_1}}, 1]$. 
\end{comment}

In the general case  $f'=  \sum_{i=0}^{s_2 -1} C_i^- v_i -   \sum_{i=0}^{s_1-1} C_i^+ u_i $, reducing to a common denominator, and denoting by $X(j):= \prod_{l=0}^{s_1-1} {x_l(j)}  $,   $Y(j):=  \prod_{l=0}^{s_2-1} {y_l(j)}  $  and by  $X_i(j):=  \prod%_{l\neq i=1}^{s_1-1}
_{\begin{subarray}{c} 1\leq 1\leq s^1-1 \\ l\neq i \end{subarray}} {x_l(j)}  $,   $Y_i(j):=  \prod_{\begin{subarray}{c} 1\leq 1\leq s^2-1 \\ l\neq i \end{subarray}} {y_l(j)}  $,  we get
%\begin{comment}
%\bes
% \left|\BS{f'}{r_l}(z_0)\right| \leq   \sum_{j=0}^{r_l-1} \frac{  \sum_{i=0}^{s-1}   \left|{C_i^-}  \sum_{l\neq i} %{y_l(j)} \prod_{l=0}^{s-1} {x_l(j)}  - {C_i^+} \prod_{l\neq i } {x_l(j)} \prod_{i=0}^{s-1} {y_l(j)}\right| } {  \prod_{l=0}^{s-1} {x_l(j)} \prod_{l=0}^{s-1} {y_l(j)}  } \right| 
%\ees
%\end{comment}
\be \label{commdenomcanc}%|
\BS{f'}{r_l} (z_0)%|
 =  \sum_{j=0}^{r_l-1} %\left|
  \frac{  \sum_{i=0}^{s_2-1}  {C_i^-}  Y_i(j ) X(j)  - \sum_{i=0}^{s_1 -1} {C_i^+}  Y(j)   X_i(j) % \right) 
  } {  X(j) Y(j)  } .%\right| 
%= \sum_{j=0}^{r_l-1}   \sum_{i=0}^{2s-1}  \left( {C_i^-}   - {C_i^+} \right) \left( \left( C \lambda^{(n_{k_0})}_{j_0} \right)^{2s-1} + O(\lambda^{(n_{k_0})}_{j_0}^{s-1+\gamma}  )  \right) } {  X(j) Y(j)  } \right| 
%eq O() \lambda^{s-1} 
\ee
%where to estimate the numerator we used that 
By Corollary \ref{deviationsxjyj}, recalling that $\|B_{k_l} \|\leq const $, we have% if we denote by $\Theta_k^\pm(j):= s \|B_k \|^{\gamma} \prod_{i=1}^s \Th{i}{k_i^{\pm}(j)}$,
\begin{eqnarray}\nonumber X(j) %= Y(j)
 &=& \left(  \frac{\lambda^{(c_l)}_{j_0}}{ {\delta^{(c_l)}_{j_0}}}  \right)^{s_1} \left( j^{s_1}+ \sum_{m_1=1}^{s_1} j^{{s_1}-m_1 +\gamma m_1 } \binom{s_1}{m_1}  \sum_{1\leq i_1<\dots <i_{m_1}\leq s_1}  \Th{k_l}{k_{i_1}^{+}(j)} \dots \Th{k_l}{k_{i_{m_1}}^{+}(j)}  \right),  \\ 
 %O\left( \Theta_{k_l}^\pm(j)   j^{s-1+\gamma})  \right), \quad  \\ 
 \nonumber X_i(j)%=Y_i(j) 
 &=&  
  \left(  \frac{\lambda^{(c_l)}_{j_0}}{{\delta^{(c_l)}_{j_0}}}  \right) ^{s_1-1} 
  \left( j^{s_1-1} + \sum_{m_2=1}^{s_1-1}  j^{s_1-1- m_2  +\gamma m_2 } \binom{s_1}{m_2} \sum_{ \begin{subarray}{c} i_1<\dots <i_{m_2} \in \\ \{1,\dots, s_1\}\backslash \{i \} \end{subarray}}  \Th{k_l}{k_{i_1}^{+}(j)} \dots \Th{k_l}{k_{i_{m_2}}^{+}(j)} \right) 
% \left( j^{s-1}+ O( \Theta_{k_l}^\pm(j) j^{s-2+\gamma})  \right). 
\end{eqnarray}
\noindent and analogous expressions hold for $Y(j)$ and $Y_i(j)$ with $ \Th{i}{k_i^{-}(j)}$ and $s_2$ instead than  $\Th{i}{k_i^{+}(j)}$ and $s_1$.
%$X(j)=Y(j)= C^s \left( \lambda^{(n_{k_0})}_{j_0} \right)^s ( j^s+ O\left( j^{s-1+\gamma})  \right)$ and $X_i(j)=Y_i(j)= $ 
%$ ( C \lambda^{(n_{k_0})}_{j_0} ) ^{s-1} \left( j^{s-1}+ O( j^{s-2+\gamma})  \right)$. 
%$$X(j)=Y(j)= \left(C \lambda^{(n_{k_0})}_{j_0} \right)^s \left( j^s+ O\left( j^{s-1+\gamma}\right)  \right), \qquad 
%X_i(j)=Y_i(j)= \left( C \lambda^{(n_{k_0})}_{j_0}  \right)^{s-1} \left( j^{s-1}+ O\left( j^{s-2+\gamma}\right)  \right).$$ 
Thus, since the coefficients of $j^{s_1+s_2-1}$ in $Y(j ) X_i(j)$ and $Y_i(j ) X(j)$ are the same,  by symmetry of the constants, the main order in $j$  of the numerator of the RHS of (\ref{commdenomcanc}) cancels out. 
Moreover, as before we have $X(j)\geq (j \lambda^{(c_l)}_{j_0})^{s_1} , Y(j) \geq (j \lambda^{(c_l)}_{j_0})^{s_2}$ because of the minimum distance between points. Thus, from balance (\ref{balancedcomparisons}) we can estimate $(\lambda^{(b_{c_l})}_{j_0})^{s_1+s_2-1}/(\lambda^{(b_{ c_l})}_{j_0})^{s_1+s_2} $ by $const \,  r_{l}$.  
Producing the lower order terms,  for each $1 \leq m_1 \leq s_1 $ and $1 \leq m_2 \leq s_2 -1 $ (or $1 \leq m_1 \leq s_1 -1 $ and $1 \leq m_2 \leq s_2 $), 
%from $j^{s_1-m_1+\gamma m_1 }j^{s_2-1-m_2+\gamma m_2}/j^{s_1+s_2}$,
 we are left with a bounded number of terms to estimate. Each one, after simplifying the power of $j$ which is $(s_1-m_1+\gamma m_1) + (s_2-1-m_2+\gamma m_2)$ at numerator and $s_1+s_2$ at denominator, is of the form
\bes
\sum_{j=1}^{r_l-1} %\frac{ (\lambda^{(b_{k_l})}_{j_0}  j)^{ 2s-(2m-1) +\gamma(2m-1_}
\frac{  %\sum{i_1, \dots, i_m} \sum_{i_1, \dots, i_{m-1} } 
 \Th{k_l}{k_{1}^{+}}(j) \dots \Th{k_l}{k_{m_1}^{+}}(j) \Th{k_l}{k_{1}^{-}}(j)\dots
   \Th{k_l}{k_{{m_2}}^{-}}(j) }{j^{1+m_1+m_2 - (m_1+m_2)\gamma }}
   \leq    \sum_{j=1}^{r_l-1} \frac{ \Th{k_l}{k_{i_0 }^{\pm}} (j)}{ j^{2-\gamma}}   
{   \left(  \max_{\begin{subarray}{c} 0\leq i \leq s-1 \\ 1\leq j < r_l  \end{subarray}} 
\left( \frac{ \Th{k_l}{k_{i }^{\pm}(j)}}{j^{1-\gamma}}\right)   \right)}^{m_1+m_2-1} 
% %\prod_{n=1}^{m}\left( \sum_{j=1 }^{r_l-1}\frac{ \Th{k_l}{k_{i_{n}}^{+}(j)}}{ j^{2-\gamma}}\right)\prod_{n=1}^{m-1}\left( \sum_{j=1 }^{r_l-1}\frac{ \Th{k_l}{k_{i_{n}}^{-}(j)}}{ j^{2-\gamma}}\right),
\ees
\noindent where $k_{i_0}$ is any index among $k_{1}^{+}, \dots , k_{m_1}^{+} $, $k_{1}^{-}, \dots , k_{m_2}^{-} $. Let us conclude the proof by showing that each of these terms is bounded (uniformely in $l$). Let $I^k_{i^+}$ (respectively $I^k_{i^-}$) be the interval of the partition $\phi^{(b_k)}$ which has $\overline{z}^+_i$ as left endpoint (respectively $\overline{z}^-_i$ right endpoint). The sum over $j$ is estimated exactly as before,  
\noindent 
%We can estimate the maximum by a constant independent on $l$, by estimating each of the sums over which the maximum is taken in the same way %in (\ref{manysinggestimate})
% as the sums with the terms $\Th{k_l}{k^{-}_0}(j)$ and  $\Th{k_l}{k^{+}_0}(j)$, i.e.~%  if $i$ is the index which realizes the maximum,
  decomposing the sum using the nested intervals  $I^{k-1}_{{i_0}^\pm}\backslash I^{k}_{{i_0}^\pm  }$. To estimate the maximum, we remark that if $x_i(j)$ (respectively $y_i(j)$) belongs to $I^{k-1}_{{i}^+}\backslash I^{k}_{{i}^+  }$ (respectively $I^{k-1}_{{i}^-}\backslash I^{k}_{{i}^- }$), then   $k_i^{+}(j)= k$ (or respectively $k_i^{+}(j)= k$) and $j \geq const \, d^{k_l-k}$ (see (\ref{jestimate})), so that
\bes
\frac{\Th{k_l}{k_i^{\pm}(j)} }{ j ^{1-\gamma}} \leq \frac{\Th{k_l}{k} }{ d ^{\overline{\gamma}(k_l-k)}} \leq \sum_{n=0}^{k_l-k}\frac{ \norm{B_{k+n-1}}} { d ^{\overline{\gamma}(n+k_l-k)}} = \sum_{m=0}^{k_l-k} \frac{\norm{B_{k_l-m-1}}} { d ^{\overline{\gamma}(2(k_l-k)-m)}} \leq \frac{const}{d^{2\overline{\gamma}(k_l-k)}} \sum_{m=0}^{k_l-k} \left( \frac{2d^2}{3} \right)^{\overline{\gamma}m}, 
\ees    
where, reasoning as before, we changed the indexes by $m=k_l-k-n $ and used (\ref{subexpgrowtheql}). Since the sum in the last expression is $O\left(  d^{2\overline{\gamma}(k_l-k)}\, (2/3)^{\overline{\gamma} (k_l-k)}\right)$, we get a uniform bound for all $0\leq k \leq k_l$, which concludes the proof.
% expression is uniformely bounded by a converging geometric series, concluding the proof. 
 %  and $I^{k-1}_{i^-}\backslash I^{k}_{i^-  }$, wher This gives an upper  bound of the form $ const \, r_{l_k}$, concluding the proof. 
\end{proofof}

\subsection{Decomposition into Birkhoff sums along towers.}\label{decompositionsec}
From the estimate of Birkhoff sums along a tower given by Proposition \ref{boundSf'growthprop}, let us derive an estimate for more general Birkhoff sums.  
\begin{prop}\label{generalsumprop}
For a.e.~$T$, there exist a constant $M'$ and sequence of induction times $\{ n_k\}_{k\in \mathbb{N}}$ such that, whenever $z_0 \in I^{(n_{k})}_{j_0}$  for some $k$ and $0< r \leq  h^{(n_{k})}_{j_0}$, we have
\be \label{generalsum}
\left| \BS{f'}{r} (z_0) \right| \leq  M'   r +  \sum_{i=0}^{s_1-1}  \frac{C_i^+}{x_i^{min}} +  \sum_{i=0}^{s_2-1} \frac{C_i^-}{y_i^{min}}  ,
\ee
where $x_i^{min}$ and $y_i^{min}$ are the closest points to the singularities defined in (\ref{minx},\ref{miny}).
%$x_i^{min}= \min_{0\leq i < r} (T^i x_0 -\overline{z_i})^+$ and $x_i^{min}== \min_{0\leq i < r} ( \overline{z_i} -T^i x_0)^+ $ are the .  
 %$x_m = \min_{0\leq i < r} T^i x_0$ and 
%$y_m = 1-\max_{0\leq i < r} T^i x_0$ are the closest points to $0$ and $1$ respectively.
\end{prop} 
\noindent 
Comparing Proposition \ref{boundSf'growthprop} below with Proposition \ref{generalsumprop}, the difference is that the time $r$ considered is here any $0\leq  r \leq  h^{(n_k)}_{j_0}$.
%\noindent Remark that here, compared to Propositon \ref{boundSf'growthprop}, the time $r$ considered is any $0\leq r\leq  h^{(n_k)}_{j_0}$.
\begin{comment}
In the proof of the Proposition we use the following notation. Let us denote the orbit segment $\{ z_0, T z_0, \dots, T^{r-1} z_0 \}$ by $\Orb{z_0}{r} $. On  $\Orb{z_0}{r}$ introduce an ordering $\prec$ and a distance $d_{\Or}$ using the natural ordering given by $T$ as follows. If $z_1, z_2 \in\Orb{z_0}{ r} $ and $z_1= T^{i_1} z_0$,  $z_2= T^{i_2} z_0$ for $ i_1, i_2 >0$, let $z_1 \prec z_2 $ iff $ i_1 < i_2$ and  let  $ d_{\Or}(z_1, z_2 ) = k $  iff $|i_1-i_2 |= k$.
\end{comment}
\begin{proof} 
Let $\{ c_l \}_{l\in \mathbb{N}}$ be the sequence associated to a.e.~$T$ in Proposition \ref{boundSf'growthprop} and let $B$ the induced cocycle defined at the end of \S \ref{powerformsec}.  Let us denote by $C$ the accelerated cocycle over the first return map of $\Z$ to $E_C$ so that $C_l : =B^{(c_l,c_{l+1})} $. 
Let $E_D \subset E_C$ be given  by Lemma \ref{subexpgrowthlemma} for $\epsilon=\ln(d/2)$. For a.e.~$T$ we can assume that the chosen lift $\hat{T}$ is recurrent to $E_D$ along a subsequence $\{ n_k:= c_{l_k}\}_{k\in \mathbb{N}} $. Then, by Lemma  \ref{subexpgrowthlemma}, we have
\be\label{subexpgrowthC}
\| C_{l_k-l}\|\leq \overline{C}_3 (d/2)^l, \qquad \forall \, 0\leq l < l_k. 
\ee
Without loss of generality, we can again assume that $g=0$ as in the proof of Proposition \ref{boundSf'growthprop}.  Moreover, we can  assume that\footnote{If not, since $I^{(c_{l_{k'}})}$ are nested, we can define  $j_{k'}$ for $0\leq k'<k$  such that $z_0 \in I^{(c_{l_{k'}})}_{j_{k'}}$. Since $h^{(c_{l_{k'}})}_{j_{k'}}$ are increasing in $k'$, we can then substitute $k$ with the unique $k'$ for which $ h^{(c_{l_{k'}-1})}_{j_{k'-1}} < r <  h^{(c_{l_{k'}})}_{j_{k'}}$.} $ h^{(c_{l_k-1})}_{j_1} \leq  r \leq  h^{(c_{l_k })}_{j_0}$ for some $j_{1}$ and $z_0 \in I^{(c_{l_k-1})}_{j_1}$. 

%Let  $\{ n_k \}$ be the same sequence in Proposition \ref{boundSf'growthprop}.
Let us use the following notation. Let us denote the orbit segment $\{ z_0, T z_0, \dots, T^{r-1} z_0 \}$ by $\Orb{z_0}{ r} $. On $\Orb{z_0}{r}$ introduce an ordering $\prec$ and a distance $d_{\Or}$ using the natural ordering given by $T$ as follows. If $z_1, z_2 \in\Orb{z_0}{ r} $ and $z_1= T^{i_1} z_0$,  $z_2= T^{i_2} z_0$ for $ i_1, i_2 >0$, let $z_1 \prec z_2 $ iff $ i_1 < i_2$ and  let  $ d_{\Or}(z_1, z_2 ) = k $  iff $|i_1-i_2 |= k$.
  
Let us decompose the orbit  $\Orb{z_0}{ r}$ into Birkhoff sums along towers as follows.
Consider first  Birkhoff sums along the towers of  $\phi^{(c_{l_k -1})}$. Let $z^{(l_k -1 )}_j$, for $0\leq j \leq a_{l_k-1}$, be the elements of $\Orb{z_0}{ r} $ which are contained in $I^{(c_{l_k -1})}$ in increasing order w.r.t. $\prec$. More precisely, define by induction $z_0^{(l_k-1) }=  z_0 $ and $z_{j+1}^{(l_k-1) } = T^{(c_{l_k-1})} z_j ^{(l_k-1) }$, 
% where $j$ is such that $z_i ^{(l_k-1)} \in I^{(c_{l_k-1})}_i(j)$, 
 so that  $z_{j+1}^{(l_k-1 )}$ is the smallest $z\succ z_{j}^{(l_k-1 )}$ such that $z\in \Orb{z_0}{ r} \cap I^{(c_{l_k-1})}$. The last step of the induction is $a_{l_k-1}$ where $a_{l_k-1} = \max \{ j \, | \,z_{j}^{(l_k-1) } \prec T^{r}z_0 \}$. Let us also define $r^{(l_k-1)}_j = d_{\Or} (z_0,  z_{j}^{(l_k-1) } )$. 
Since $r\geq  h^{(c_{l_k-1})}_{j_1}$, $a_{l_k-1}\geq 1$. Moreover, since  $d_{\Or}(z_{j}^{(l_k-1)},  z_{j+1}^{(l_k-1) } ) \geq \min_l h^{(c_{l_k-1})}_{l}$, %from the fact that $T$ is bounded type and ,
  using balance (\ref{balance}) and (\ref{heightsrelation}) and $r\leq  h^{(c_{l_k})}_{j_0}$, we also have that $a_{l_k-1} \leq  r/ \min_l h^{(c_{l_k-1})}_{l} \leq \nu \| C_{l_k}\|$. So far we can write
\bes
 \BS{f'}{r} (z_0)  = \sum_{j=0}^{ a_{l_k-1}-1} \BS{f'}{r^{(l_k-1)}_{j+1} - r^{(l_k-1)}_j } 
%{d_{\Or} (  z_{i}^{(k_0-1) }, z_{i+1}^{(k_0-1) } ) }
{ (z_{j}^{(l_k-1)})} +  \BS{f'}{r'} { (z_{a_{l_k-1} }^{(l_k)})}  , \qquad %r'= r-r^{(l_k-1)}_{a_{l_k-1}} , \quad
 a_{l_k-1} \leq \nu \| C_{l_k}\|, 
\ees
where $r'= r-r^{(l_k-1)}_{a_{l_k-1}} $ and each term in the sum is by construction a Birkhoff sum along a tower of $\phi^{(c_{l_k -1})}$ %f order $c_{l_k -1}$
 while the last term is a reminder that cannot be decomposed any more into  Birkhoff sums along towers of the same order. 

Let us continue  by induction to decompose the reminder into  Birkhoff sums along the towers  of $\phi^{(c_{l})}$ with $0\leq l< l_k-1$. To get from step $l+1$ to step $l$, let $z_0^{(l) }=  z_{a_{l+1}}^{(l+1)}$ and $z_{j+1}^{(l) } = T^{(c_{l})} z_j ^{(l) } 
 \in I^{(c_{l})}$
%T^{h^{(c_{l})}_j} z_i ^{(l) }$ where $j$ is such that $z_i ^{(l)} \in I^{(c_{l})}_j$, 
for $j=0, \dots, a_l$ with  $a_l = \max \{ j \, | \,z_{j}^{(l) } \prec T^{r} z_0 \} $. In this way again %$z_0^{(l) }\prec
$ z_1^{(l) } \prec \dots \prec z_{a_l}^{(l) }$ are all elements $z\in \Orb{z_0}{ r}$ with $z\succ  z_{a_{l+1}}^{(l+1)}$  for which $ z\in I^{(c_{l})}$. Letting $r^{(l)}_j = d_{\Or} (z_0,  z_{j}^{(l) } )$, we have
\be \label{BSdecom}
 \BS{f'}{r} (z_0)  = \sum_{l=0}^{l_k-1} \sum_{j=0}^{ a_{l}-1} \BS{f'}{r^{(l)}_{j+1} - r^{(l)}_j } 
{( z_{j}^{(l)}) } +  \BS{f'}{r'} ({ z_{a_{0} }^{(0)}})  , \qquad r'= r-r^{(0)}_{a_{0}}\leq \max_l h_l^{(c_{0})} ,
\ee
where, if $a_l=0$ (it might happen for $l<l_k-1$), the sum over $j$ is taken by convention to be zero.
Moreover, as before, by construction we have $a_{l} \leq  \nu \| C_{l+1} \|$, $0\leq l\leq l_k -1$.

\begin{comment}
Remark that, by construction, $ z_{j}^{(l)} \in I^{(c_{l})} \backslash I^{(c_{l+1})}$ for all $j=1, \dots, a_l$, where $l\leq l_0-1$. Hence,
\be \label{geometricdistances}
 z_{j}^{(l)} \geq  \lambda^{(c_{l+1})}, \qquad  j=1, \dots, a^{(l)};
\ee
\end{comment}

Let us apply Proposition \ref{boundSf'growthprop} to each addend in the double sum in (\ref{BSdecom}), denoting, in each  Birkhoff sum along a tower, the points which are closest to right and left singularities (recalling that $( \, \cdot \, )^{pos}$ denotes the positive part) by 
% $\overline{z}_j$ (for $j=0,\dots, s$) by 
%$ y_{i}^{(k)} = \min_{r^{(k)}_{i+1} \leq i < r^{(k)}_i  } (1-T^i x_0)$ the closest point to $1$ in each special Birkhoff sum, 
\begin{eqnarray}
&& ( x^{min}_i)^{(l)}_j %= \min_{r^{(l)}_{i} \leq s < r^{(l)}_{i+1}  } (\overline{z}_i-T^s z_0)^+
 = \min _{0\leq s < r^{(l)}_{j+1} - r^{(l)}_j  }
(T^s z^{(l)}_i  -\overline{z}_i^+ )^{pos} ,  \qquad i=0, \dots, s_1-1 ;\nonumber \\
&&  ( y^{min}_i)^{(l)}_j = \min _{0\leq s < r^{(l)}_{j+1} - r^{(l)}_j  }
(\overline{z}^-_i  -T^s z^{(l)}_i  )^{pos}, \qquad i=0, \dots, s_2-1.
%, \quad  ( y^{min}_0)^{(l)}_i:= ( y^{min}_s)^{(l)}_i .
 \nonumber
\end{eqnarray}
%and remarking that by construction, since $( x^{min}_j)^{(l)}_i  $ is the only point of the special Birkhoff sum orbit that belong to $I^{(c_l)}$,  we have $( x^{min}_j)^{(l)}_i = z_{i}^{(l)}$. 
We get 
%\begin{eqnarray} \nonumber 
\be\label{estimateBSdecomp1}
\left| \BS{f'}{r} (z_0)  \right| \leq \sum_{l=0}^{l_k-1} \sum_{j=0}^{ a_{l}-1} \left(  \sum_{i=0}^{s_1-1}  \frac{C_i^+}{ ( x^{min}_i)^{(l)}_j} +
 \sum_{i=0}^{s_2-1}  \frac{C_i^-}{  ( y^{min}_i)^{(l)}_j } +  M  ({r^{(l)}_{j+1} - r^{(l)}_j }) \right)  +  \left| \BS{f'}{r'} { (z_{a_{0} }^{(0)})} \right| .
\ee
%%\\ & \leq & \label{estimateBSdecomp}
% \sum_{l=0}^{k_0-1} \sum_{i=0}^{ a_{k}-1}  \frac{1}{ z_{i}^{(k)}} 
%  \frac{1}{z_0} +  \sum_{i=1}^{ a_{k}}  \frac{1}{ z_{i}^{(k)}}
% +  \nu \overline{A} \sum_{k=0}^{k_0-1}  \frac{1}{ ( y_{m}^{(k)})_i}  +  M \sum_{k=0}^{k_0-1}  r^{(k)}_{a_k} +  {r'} \left( \frac{1}{z_m} +% \frac{1}{y_m}\right),
%\end{eqnarray}
%To estimate the first term in (\ref{estimateBSdecomp1}), let us rearrange the points $ z_{i}^{(k)}$, $ i=0, \dots, a_{k}-1$ using that $z_0^{(k) }=  z_{a_{k+1}}^{(k+1)}$ for $k=0,\dots\, a_k$, so that
% $\sum_{k=0}^{k_0-1} \sum_{i=0}^{ a_{k}-1}   \frac{1}{ z_{i}^{(k)}} =\frac{1}{z_0}+\sum_{k=0}^{k_0-1} \sum_{i=1}^{ a_{k}}  \frac{1}{ z_{i}^{(k)}} - \frac{1}{ z_{a_0}^{(0)}}$ 
%\be\label{rearrange}
%\sum_{k=0}^{k_0-1} \sum_{i=0}^{ a_{k}-1}   \frac{1}{ z_{i}^{(k)}} =\frac{1}{z_0}+\sum_{k=0}^{k_0-1} \sum_{i=1}^{ a_{k}}  \frac{1}{ z_{i}^{(k)}} - \frac{1}{ z_{a_0}^{(0)}} %\leq \frac{ \nu \overline{A}}{z_0} + \sum_{k=0}^{k_0-1} \sum_{i=1}^{ a_{k}}  \frac{1}{ z_{i}^{(k)}} 
%\ee
%and, since $ z_{i}^{(k_0-1)}\geq z_0$, let us estimate the terms of order $k_0-1$ as  $\left|\sum_{i=1}^{ a_{k_0-1}} \frac{1}{ z_{i}^{(k_0-1)}}\right|\leq \frac{ \nu \overline{A}}{z_0} $. % where the terms of order $k_0-1$ where estimated using that $z_0\leq  z_{i}^{(k_0-1)}$.
% can be estimated as $\sum_{i=1}^{ a_{k_0-1}} \frac{1}{ z_{i}^{(k_0-1)}}\leq \frac{ \nu \overline{A}}{z_0} $. 
%Estimating also the last two terms of  using  that the 
Since the  sum $ \sum_{l=0}^{l_k-1}\sum_{j=0}^{ a_{l}-1} (r^{(l)}_{j+1} - r^{(l)}_j ) $ is telescopic, it reduces to $ r^{(0)}_{a_0} \leq r$. Moreover the last term in (\ref{estimateBSdecomp1}) can be estimated by $M r'+ \sum_{i=0}^{s_1-1}  \frac{C_i^+}{   x_i^{min}  } +   \sum_{i=0}^{s_2-1} \frac{C_i^-}{y_i^{min}} $ with $r'\leq  \max_j h_j^{(c_{0})}$, which is a constant.
%\bes
% \sum_{k=0}^{k_0-1}\sum_{i=0}^{ a_{k}-1} ({r^{(k)}_{i+1} - r^{(k)}_i })  \sum_{k=0}^{k_0-2}( r^{(k)}_{a_k} -   r^{(k)}_{a_k}) + r^{(k_0-1)}_{a_{k_0-1}}  = r^{(0)}_{a_0} \leq r 
%\ees
% which is a constant, both the last two terms in (\ref{estimateBSdecomp}) satisfy already a bound of the desired form of (\ref{generalsumprop}). 
%\be \label{estimateBSdecomp}
%\left| \BS{f'}{r} (z_0)  \right| \leq  
%  \frac{ \nu \overline{A} + 1}{z_0} +   \sum_{k=0}^{k_0-2} \sum_{i=1}^{ a_{k}}  \frac{1}{ z_{i}^{(k)}}
% +  \nu \overline{A} \sum_{k=0}^{k_0-1}  \sum_{i=0}^{ a_{k}-1} \frac{1}{ ( y_{m}^{(k)})_i}  +  M r   
%\ee
Hence,  to conclude the proof of (\ref{generalsum}),   %both the last two terms in (\ref{estimateBSdecomp}) satisfy a bound of the desired form ofand 
we  are left to estimate the sums of contributions from closest points to the singularities in each cycle. % of contributions coming from  closest points  $\{ x_{i}^{(k)}} \}$ and  $\{ ( y_{m}^{(k)})_i \}$ 
Let us show that their contributions decrease exponentially in the order $k$ of the towers, thanks to the choice (\ref{subexpgrowthC}) of the times $\{c_{l_k}\}_{k\in \mathbb{N}}$.

%Given any $0\leq j\leq s-1$,  Let us consider, given any $0\leq j\leq s-1$,  the contribution to (\ref{estimateBSdecomp1}) coming from the points  
Given any $0\leq i\leq s_1-1$, 
Let us consider the contribution to (\ref{estimateBSdecomp1}) coming from the points  
\be\label{grouplevell}
\left\{ {  ( x^{min}_i)^{(l)}_j }, \qquad   j=0, \dots, a_l-1 \right\} .
\ee   
%The contribution from ${  ( y^{min}_i)^{(l)}_i }$, $ i=1, \dots, a^{(l)}$ is estimated in the same way.
Assume first that $0\leq l<l_k-1$.  We remark that all these points belong by construction to a unique tower of order $l+1$, the tower $Z^{(c_{l+1})}_{j(l+1)}$ such that $z^{(l)}_0 = z^{(l+1)}_{a_{l+1}+1}  \in I^{(c_{l+1}) }_{j(l+1)}$. Thus, the minimum spacing between them is by balance (\ref{balance}) at least $\lambda^{(c_{l+1})} /d\nu$. Thus, if we consider separately the minimum of (\ref{grouplevell}),  each of the other $a_l-1$ points of  (\ref{grouplevell})  gives a contribution less than  $C_i^+ d \nu / \lambda^{(c_{l+1})} $. Since the minimum of (\ref{grouplevell}) is bigger than the minimum $( x^{min}_i)^{(l+1)}_{a_{l+1}+1} $ of the level ${l+1}$ orbit segment of length $h^{(c_{l+1})}_{j(l+1)}$ %(contained in the tower $Z^{(c_{l+1})}_{i_{l+1}}$)
 which contains  all points in (\ref{grouplevell}), it can be included in the analogous estimate corresponding to $l+1$, by considering $a_{l+1}$ contributions equals to $C_i^+ d \nu / \lambda^{(c_{l+2})} $ instead than $a_{l+1}-1$. 
When $l=l_k-1$, the minimum is simply given by $x_i^{min}$ and the contributions of all the other points are again estimated by $C_i^+ d \nu / \lambda^{(c_{l_k})} $ since they are all contained in different floors of the tower $Z^{(c_{l_k})}_{j_0}$. 

Hence, first recalling that $a_l \leq \nu \| C_{l+1}\|$, then using that $\lambda^{(c_{l+1})} \geq d^{l_k-l-1} \lambda^{(c_{l_k})}$ and setting $l':=l_k-l-1$, we get 
\be
 \sum_{l=0}^{l_k-1} \sum_{j=0}^{ a_{l}-1}   \frac{C_i^+}{ (x^{min}_i)^{(l)}_j} \leq \frac{ C_i^+}{  x^{min}_i} +  \sum_{l=0}^{l_k-1}\nu \| C_{l+1}\|  \, \frac{ C_i^+ d\nu }{\lambda^{(c_{l+1})}} 
 \leq \frac{ C_i^+}{  x^{min}_i} +  \frac{C_i^+ d\nu^2  }{\lambda^{(c_{l_k})}}\sum_{l'=0}^{l_k-1} \, \frac{\| C_{l_k-l'}\| }{d^{l'}} , 
\ee
where the last series is uniformly bounded for all $k$ by (\ref{subexpgrowthC}). Since by balance (\ref{balance}, \ref{balancedcomparisons}) we have  $1/\lambda^{(c_{l_k})} \leq \nu^2 \, h^{(c_{l_k})}_{j_0}$ and $h^{(c_{l_k})}_{j_0} \leq \| C_{l_k}\| \nu^2 h^{(c_{l_k-1})}_{j_1} \leq \overline{C}_3\nu^2  r$ (where the last inequality uses again (\ref{subexpgrowthC})), we get a bound of the desired form. 

Since, for any $0\leq i \leq s_2-1$, the contribution from ${  ( y^{min}_i)^{(l)}_j }$, $ j=0, \dots, a_l-1$  is estimated in the same way, this concludes the proof. 
\end{proof}

\subsection{Birkhoff sums variations and proof of absence of mixing.}\label{proofabsencethmsec}
In this Section we  complete the proof of Theorem \ref{absencethm}. 

\begin{proofof}{Theorem}{absencethm}
%In order to prove absence of mixing, 
Let us verify the assumptions of the criterion for absence of mixing (Lemma \ref{absencemixingcriterion}). Given a typical $T$, consider the subsequence  $\{n_k\}_{k\in \mathbb{N}}$ of balanced times for which Proposition \ref{generalsumprop} holds.  In \S \ref{rigiditysec}  we already defined starting from $\{n_k\}_{k\in \mathbb{N}}$ a corresponding class of sets $E_k$ and times $r_k$ which verify conditions $(i)$ and $(ii)$ of the Definition \ref{rigiditysetsdef} of rigidity sets. Let us hence define the subintervals $J_k$ introduced  in \S \ref{rigiditysec} in such a way that also  condition $(iii)$ is satisfied.  
Referring to the notation in \S \ref{rigiditysec}, if $(I^{(n_k)}_{j_0} )_{ l_0}= [a,b)$, and $\lambda = b-a$ is its length, define $J_k = [ a + \lambda/4 , b - \lambda/4[$.  %Consider as $\{n_k\}$ only the subsequence of times for which Proposition \ref{boundSf'growthprop} holds. 

Let us consider any two points $y_1, y_2 \in E_k$. By definition of $E_k$, we can write $y_1 = T^{k_1}z_1$, $y_2 = T^{k_2} z_2$ where $z_1, z_2 \in J_k$ are two points in the base and $0\leq k_1, k_2 < h^{(n_k)}_{j_0 }$. In order to prove $(iii)$, let us split the Birkhoff sums as 
\be\label{variation3parts}
\begin{split}
 \BS{f}{r_k} (y_1) &- \BS{f}{r_k} (y_2)  = \left( \BS{f}{r_k}  (T^{k_1} z_1) - \BS{f}{r_k} (z_1) \right)+\\ &+ \left( \BS{f}{r_k} (z_1) - \BS{f}{r_k} (z_2)\right) + \left( \BS{f}{r_k} ( z_2)  - \BS{f}{r_k} ( T^{k_2} z_2) \right) .
\end{split}
\ee 
Let us estimate the central term first. Since $J_k$ is  by construction contained in a continuity interval of $T^i$ for $0\leq i < r_k$, the mean value theorem gives that there exists some $z_0$ in $J_k$ such that 
\be \label{meanvaluecentral}
\left| \BS{f}{r_k} (z_1) - \BS{f}{r_k} (z_2) \right| \leq | \BS{f'}{r_k} (z_0)| |z_1-z_2| \leq  | \BS{f'}{r_k} (z_0)| \lambda/2 .
\ee 
The return time $r_k$ is in particular a return time to $I^{(n_k)}_{j_0}$ and hence to $I^{(n_k)}$; assume it is the $n^{th}$ return to $I^{(n_k)}$, where $n=\sum_{j=1}^d{a^k_j}$  and $a^k_j$ is the number of returns to $I^{(n_k)}_{j}$ (so $a^k_{j_0}\geq 1$). Hence we can write 
%of return times to $I^{(n_k)}$ needed to get $r_k$.let us show that it can be written as 
$r_k =\sum_{j=1}^{d}  a^k_j h^{(n_k )}_j $. 
Let us decompose the Birkhoff sums into  Birkhoff sums along the towers of $\phi^{(n_k)}$  as 
\be \label{decomp}
 \BS{f'}{r_k} (z_0)  = \sum_{j=1}^{d} \sum_{l=1}^{a^k_j}  \BS{f'}{ h^{(n_k )}_j} (z_l^j) ,
\ee
where the intermediate points\footnote{To construct the intermediate points and show (\ref{decomp}) one can use induction on $n$. 
 %let us procede by induction on the cardinality $n=\sum_{j=1}^d{a^k_j}$ of return times to $I^{(n_k)}$ needed to get $r_k$. 
If $n=1$, then
%as follows. Consider the first return time $h^{(n_k)}_{j_0}$ of $z_0$ to $I^{(n_k)}$  If 
$T^{h^{(n_k)}_{j_0}}z_0 \in J_k$ and there is nothing to prove, since $r_k=h^{(n_k)}_{j_0}$ and setting $a^k_{j_0}=1$ and $a^k_j = 0$ for $j\neq j_0$,  (\ref{decomp}) becomes an identity.
Otherwise, if $n>1$, let $j$ be such that  $T^{h^{(n_k)}_{j_0}}z_0 \in I^{(n_k)}_{j}$ and define $z_j^1 = T^{h^{(n_k)}_{j_0}}z_0 $ so that $ \BS{f'}{r_k} (z_0) = \BS{f'}{ h^{(n_k) }_j} (z_0) + \BS{f'}{r'} (z_j^1)$, with $r'=r_k-h^{(n_k )}_j$. Since for $r'$ the number of return times is by construction $n-1$, the definition of the remaining intermediate points $z_l^j$ follows by induction.} are $z_1^{j_0}=z_0 $ and $z_l^j \in I^{(n_k)}_{j}$ for $l=1,\dots,a^k_j$ and for $j\neq j_0$ one can have $a^k_j = 0$, in which case the corresponding sum is empty. 

%To construct such intermediate points, one can use induction on $n$. 
% %let us procede by induction on the cardinality $n=\sum_{j=1}^d{a^k_j}$ of return times to $I^{(n_k)}$ needed to get $r_k$. 
%If $n=1$, then
%%as follows. Consider the first return time $h^{(n_k)}_{j_0}$ of $z_0$ to $I^{(n_k)}$  If 
%$T^{h^{(n_k)}_{j_0}}z_0 \in J_k$ and there is nothing to prove, since $r_k=h^{(n_k)}_{j_0}$ and setting $a^k_{j_0}=1$ and $a^k_j = 0$ for $j\neq j_0$, % (\ref{decomp}) becomes an identity.
%Otherwise, if $n>1$, let $j$ be such that  $T^{h^{(n_k)}_{j_0}}z_0 \in I^{(n_k)}_{j}$ and define $z_j^1 = T^{h^{(n_k)}_{j_0}}z_0 $ so that $ %\BS{f'}{r_k} (z_0) = \BS{f'}{ h^{(n_k) }_j} (z_0) + \BS{f'}{r'} (z_j^1)$, with $r'=r_k-h^{(n_k )}_j$. Since for $r'$ the number of return times is by %construction $n-1$, the definition of the remaining intermediate points $z_l^j$ follows by induction.

% and hence one can apply induction to $ \BS{f'}{r'} (z_j^1)$, for which $n$ is decreased by one, for the definition of the points $z_l^j$.
% can be then continued considering $\BS{f'}{ r_k-h^{n_k }_j} (z_j^1)$ and using induction 
%of the number $\sum_{i=1}^{d} \sum_k a^k_i h^{(n_k )}_i f cycles 

To each of the Birkhoff sums in (\ref{decomp}) let us apply Proposition \ref{boundSf'growthprop}. 
Remark that $T^i (I^{(n_k)}_{j_0} )_{ l_0}  $ for $i=0,\dots, r_k-1$ are all disjoint and rigidly translated by $T$ and that the singularities $\overline{z}^{\pm}_j$ are all contained in the boundary of the floors of the towers, so that, since $z_0$ belongs to a  central subinterval $J_k \subset (I^{(n_k)}_{j_0} )_{ l_0}  $,
%, the orbit points $T^i z_0$ all belong by definition of $J_k$ to the central parts of the corresponding $T^i (I^{(n_k)}_{j_0} )_{ l_0}  $. 
the distance of each point in each orbit segment $\{ T^i z_l^j \}_{i=0,\dots, {h^{(n_k) }_j-1}}$ % the closest points of each orbit segment $\{ T^i z_l^j \}_{i=0,\dots, {h^{(n_k) }_j-1}}$ to the endpoints $0$ and $1$ 
 from the singularities $\overline{z}^+_j$, $0\leq j \leq s_1$ and  $\overline{z}^-_j$, $0\leq j \leq s_2$ is at least $\lambda/4$. Moreover, we have $\sum_j a^k_j\leq 2d(d+2)\nu^2$, as it follows by combining that,  by disjointness  of $T^i (I^{(n_k)}_{j_0} )_{ l_0}  $, we have $\sum_j a^k_j h^ {(n_k)}_{j} \lambda/2 \leq 1$ and that, by balance (\ref{balance}) and construction (\ref{bigtower}, \ref{biginterval}) of $J_k$, we have $ h^{(n_k)}_j \geq 1/ (\lambda d(d+2)\nu )$. %,  $ h^{(n_k)}_j \geq \lambda/\nu $, which implies $\sum_j a^k_j\leq 2\nu$ .
Hence, setting $\overline{C}=\sum_{j=0}^{s_1-1} C_j^+  + \sum_{j=0}^{s_2-1} C_j^-$, we have
\be\label{centralBSf'bound}
| \BS{f'}{r_k} (z_0)| \leq 2 d(d+2)\nu \left( M  \max_j h^{(n_k )}_j +  \frac{4 \overline{C} }{\lambda}\right) .
\ee 
Equation (\ref{centralBSf'bound}) combined with (\ref{meanvaluecentral}) and using again balance (\ref{balancedcomparisons}), gives the bound on the central term of (\ref{variation3parts}) by a constant.
Each of the other two terms of (\ref{variation3parts}) can be written, for $\nu=1,2$, as
\be\label{rewritingother2}
\begin{split}
 & \BS{f}{r_k}  (T^{k_\nu} z_\nu) - \BS{f}{r_k} (z_\nu)  = ( \BS{f}{r_k-k_\nu} (T^{k_\nu } z_\nu) + \BS{f}{k_\nu}(T^{r_k} z_\nu) )   - \\ &(  \BS{f}{k_\nu} (z_\nu) + \BS{f}{r_k - k_\nu} (T^{k_\nu } z_\nu)) = 
 \BS{f}{k_\nu} (T^{r_k } z_\nu) -  \BS{f}{k_\nu}( z_\nu)=   \BS{f'}{k_\nu}( u_\nu)(T^{r_k } z_\nu -z_\nu ), 
\end{split}
\ee
where $u_\nu$ is a point in $I^{(n_k)}_{j_0}$ between $z_\nu$ and $T^{r_k} z_\nu$. Since $z_\nu  \in J_k \subset  I^{(n_k)}_{j_0}$ and,  by construction of $J_k$, $T^{r_k} z_\nu \in  T^{r_k} J_k \subset  I^{(n_k)}_{j_0}$, $u_\nu$ has distance at least  $\lambda/4$ from both endpoints of $I^{(n_k)}_{j_0}$. Moreover $T^i I^{(n_k)}_{j_0}$ for $0\leq i< h^{(n_k)}_{j_0}$ are disjoint and are rigid translates of  $I^{(n_k)}_{j_0}$, so that all the points $T^i u_\nu$ for $0\leq i\leq k_\nu$ have  distance at least  $\lambda/4$ from both endpoints of $T^iI^{(n_k)}_{j_0}$ % the points of the orbit $T^i u_s$, $0\leq i\leq k_s$  have distance at least $\lambda/4$ from both endpoints of $T^iI^{(n_k)}_{j_0}$
and  hence in particular from the singularities $\overline{z}_j^{+}$, $j=0, \dots, s_1$,  $\overline{z}_j^{-}$, $j=0, \dots, s_2$.  

We can hence apply Proposition \ref{generalsumprop} with $x_i^{min}, y_i^{min} \geq \lambda/4$ to get
\be\label{boundother2}
\left| \BS{f'}{k_\nu}( u_\nu)(T^{r_k } z_\nu -z_\nu )\right|\leq \left( M' k_{\nu } + {4\overline{C}}/{\lambda}\right)   \lambda^{(n_k)}_{j_0}\leq  M'+  8 \overline{C} (d+2),
\ee
where in the last inequality we used that $\lambda\geq  \lambda^{(n_k)}_{j_0}/2(d+2)$, by (\ref{biginterval}) and definition of $J_k$ and that $k_\nu\lambda^{(n_k)}_{j_0} \leq  h^{(n_k)}_{j_0} \lambda^{(n_k)}_{j_0} \leq 1$.
Thus, combining (\ref{rewritingother2}) and (\ref{boundother2}) we get the upper bound of the other two terms in (\ref{variation3parts}) by a constant. This concludes the  proof that $E_k$ and $r_k$ satisfy also Property $(iii)$ of Lemma \ref{absencemixingcriterion} and hence, using Lemma \ref{absencemixingcriterion},  the Proof of Theorem \ref{absencethm}.
\end{proofof}

\section{Reduction to special flows.}\label{reductionmultivalHsec}
In this section we first explain the notion of \emph{typical} that is used in Theorem \ref{multivalHthm} and then derive Theorem \ref{multivalHthm} from Theorem \ref{absencethm}.
%Consider first 
In the space of deformations of closed smooth (at least $\mathscr{C}^2$) $1$-forms by closed smooth exact $1$-forms, %where smooth is . A 
a generic one form $\omega$ %(with respect to any such perturbation) 
is \emph{Morse}, i.e. is locally given by $\omega = \ud H$ where $H$ is a Morse function. Thus, all zeros of $\omega$ correspond to either centers or simple saddles. Around each center there is an neighborhood filled by periodic orbits. Moreover, there can be cylinders filled by closed orbits (see \cite{Ma:tra, Le:feu, Zo:how}) and both type of regions of periodic orbits are bounded by saddle loops. 
We remark that such regions persist under perturbation by a closed form and hence % an open set.
% Hence, 
the set of $\omega$ with centers or cylinders of periodic orbits is open and that, in this open set of flows, mixing can be produced by the mechanism used in \cite{SK:mix, Ul:mix}.  

Let us assume that the multi-valued Hamiltonian flow associated to $\omega$ has only simple saddles and there are no saddle connections, which in particular implies that the flow is minimal. From a result by Calabi \cite{Ca:ani} or by Katok \cite{Ka:inv}, there exists a holomorphic one form (or Abelian differential) $\omega'$ whose associated vertical flow determines the same measured foliation. Thus one can define a measure-class and a corresponding measure-theoretical notion of \emph{typical} multi-valued Hamiltonian coming from the measure class on strata of Abelian differentials, which is induced by the period map as follows. Let $\{\gamma_i\}_{i=1}^n$ be a basis of the relative homology $H_1(S, \Sigma, \mathbb{C})$, where $\Sigma$ are the zeros of $\omega$ and $n=2g+|\Sigma|-1$. The period map sends $\omega$ to the vector in $\mathbb{R}^n$ whose coordinates are integrals $\int_{\gamma_i} \omega$, $i=1, \dots, n$. The pull-back of the Lebesgue measure on $\mathbb{R}^n$ gives the desired measure-class%. % (that can be renormalized so that the torus $\mathbb{T}^2g$ has measure $1$).
\footnote{Another measure-class is described by Zorich in \cite{Zo:how}. The surface $S$ can be embedded in a torus $\mathbb{\mathbb{T}^N}$ so that $\omega$ is the pull-back of a linear form $a_1 \ud x_1 \wedge \dots a_N \ud x_N$. Typical hence refers to a.e. direction perpendicular to the hyperplane given by $a_1, \dots, a_N$ in   $\mathbb{\mathbb{R}^N}$.}.   
Let us remark that the vertical flow of a typical Abelian differential $\omega'$ is a minimal flow without saddle connections. 
% (for example if periods are rationally independent and hence there are no saddle connections) gives a \emph{minimal} flow.   

\begin{proofof}{Theorem}{multivalHthm} Let $\{ \varphi_t\}_{t \in \mathbb{R}}$ be a multi-valued Hamiltonian flow on $S$. 
For any $\gamma$ transversal cross section  to the flow, the Poincar{\'e} first return map $T$ on $\gamma$ preserves the measure induced by the area form on the transversal. Up to reparametrization (using the smooth conjugacy that sends the induced invariant measure to the Lebesgue measure on an unit interval parameterizing $\gamma$) we can assume that $T$ is an IET $(\underline{\lambda}, \pi)$ on $I^{(0)}$. Since we assume that the flow is minimal, $\pi$ is irreducible. The flow $\{ \varphi_t\}_{t \in \mathbb{R}}$ is isomorphic (up to the smooth conjugacy above) to a a special flow over $T$ under a roof function $f$ which is given by the first return time to the transversal. Using the representation of typical {A}belian differentials as zippered rectangles (for which we refer for example to \cite{Yo:con} or \cite{Vi:IET}), one can see that a full measure set of IETs gives a set of typical Abelian differentials and hence typical multi-valued Hamiltonian flows. Hence, in order to deduce Theorem \ref{multivalHthm} from Theorem \ref{absencethm}, it is enough to check that $f$ satisfies the assumptions of Theorem \ref{absencethm}.

 Let us choose the transversal $\gamma$ so that the  backwards flow orbits of the transversal endpoints both contain a saddle, but the endpoints are not at saddles, for example by considering the  usual cross section chosen for the zippered rectangle representation of the corresponding Abelian differential and shifting it by a sufficiently small $t_0<0$ along the flow direction.  Both discontinuities of $T$ and singularities of $f$ occur at points $\overline{z}_i$ which lie on a separatrix, hence whose forward orbit under $\varphi_t$ limit towards a saddle.  By choice of the transversal endpoints, the IET exchanges $d= 2g+s-1$ intervals, where $g$ is the genus and $s$ the number of saddles, and since the saddles are simple, by Gauss-Bonnet formula\footnote{If $k_i$, for $i=i, \dots, s$, denote  the orders of the zeros, the Gauss-Bonnet formula gives  $\sum_{i=1}^{s} k_i = 2g-2$ (we refer for example to \cite{Yo:con}  or \cite{Vi:IET}) and since we are assuming that $k_i=1$ for all $k=1, \dots, s$, we have $s=2g-2$.} $d=4g-3=2s+1$. 
%By choice of the transversal, the IET has $2g+k-1$ discontinuities, 
%n order to deduce Theorem \ref{multivalHthm} from Theorem \ref{absencethm}, it is enough to check that $\varphi_t$ is a suspension over $T$ of the desired form. 

Since the parametrization is locally Hamiltonian, trajectories are slowed down more and more the closer they come to a saddle. The one-sided limit  $\lim _{x  \rightarrow \overline{z}_i^{+}}f(x)$ (or $\lim_{x  \rightarrow \overline{z}_i^{-}} f(x)$) of the return time $f(x)$ blows up near $\overline{z}_i$ if the forward trajectories of the nearby 
%if the right boundary of the zippered rectangle containing  points $x$  \geq \overline{z}$ 
points $x> \overline{z}_i$ (or $x < \overline{z}_i$) under the vertical flow of the Abelian differential, considered 
% the trajectories $\{ \varphi_t (x)\}_{0,\leq t\leq {t(x)}}$ through $x  which
up to their return time, come arbitrarly close to a saddle. From the canonical form of a simple  saddle,  a calculation  (see \cite{Ar:top}) shows that the singularities are in this case logarithmic, i.e.~of the form $C_i | \log (x-\overline{z}_i)| $ for  $x\geq \overline{z}_i$ (or  $x\leq \overline{z}_i$) up to a function whose derivative has bounded variation, where the constant $C_i$ depends on the saddle. 

One can see, for example using  the zippered rectangle representation, % (for which we refer to or \cite{Vi:IET}),
  that out of the $2d=4s+2$ (two for each interval) one-sided limits of the form  $\lim_{x  \rightarrow \overline{z}_i^{\pm}} f(x)$, where $\overline{z}_i$ is either a discontinuity of $T$ or an endpoint of $I^{(0)}$, exactly   $4s$ are infinite and give discontinuties of $f$, since the corresponding zippered rectangle boundary contains a saddle, while the remaining two limits are finite and the correspoding return times for nearby $x > \overline{z}_i$ or $x < \overline{z}_i$ are bounded. % since the neighbouring trajectories $\{ \varphi_t (x)\}_{0,\leq t\leq \overline{t}}$ through $x \geq \overline{z}$ do not limit to a saddle before the return time $\overline{t}  
%and the first return time to $\gamma$
%the  finite return time to $\gamma$.  
Each of the $s$ saddles has two ingoing separatrices, each of which generate a left and a right logarithmic singularity  of $f$ with the same constant $C_i$ depending on the saddle. 
%$\overline{z}_{i^1}^{+}$ and  $\overline{z}_{i^2}^-$ of $f$, on both sides $\overline{z}_{i_1}$ and $\overline{z}_{i_2}$ the logarithmic singularities have constants $C^{\pm}_{i_1} = C^{\pm}_{i_2}=C_i$. 
Thus, the number of right and left singularities is $s_1=s_2=2s$ and each constant $C_i$ appears four times, twice in a right-side singularity and twice in a left-side one. In particular, the logarithmic singularities are symmetric. 
\end{proofof}

\begin{rem}
From the proof of Theorem \ref{multivalHthm} one can see that the class of special flows which are used to represent multi-valued Hamiltonian flows is less general than the class considered in Theorem \ref{absencethm}. The permutations $\pi$ which arise are not all irreducible ones,  but only the ones which correspond to Abelian differentials in the principal stratum $\mathscr{H}(1, \dots, 1)$ of Abelian differentials with simple zeros\footnote{ For example the permutations of the form $(n \, n\!-\!1 \dots 2 1)$ %for which the arguments by Scheglov \cite{Sch:abs} can be adapted 
correspond to   a principal stratum only if $g=2$ and $n=5$.\label{n21}} and the roof functions have symmetric logarithmic singularities in which $s_1=s_2$ and all constants appear in quadruples $C_{i_1}^{+}, C_{i_2}^{+}, C_{j_1}^{-},C_{j_2}^{-} $  (for some $0\leq i_1\neq i_2 < s_1=s$, $0\leq j_1\neq j_2 < s_2=s$)  such that $C_{i_1}^{+}= C_{i_2}^{+}= C_{j_1}^{-} =C_{j_2}^{-} $.     
\end{rem}

\section*{Acknowledgements}
We would like to thank Artur Avila, Giovanni Forni and Yakov Sinai for useful and inspiring discussions. The author is currently supported by an RCUK Academic Fellowship, whose support is fully acknowledged.
Part of this work was completed while at the Institute for Advanced Studies, thanks to the support by the Giorgio and Elena Petronio Fellowship Fund and NSF grant No DMS-0635607.\footnote{This material is partially based upon work supported by the National Science Foundation under agreement No DMS-0635607. Any opinions, findings and conclusions or recommendations expressed in this material are those of the author and do not necessarily reflect the views of the National Science Foundation.}

\bibliography{biblioAbsenceMixing}
%\bibliography{bibliographyWM}
\bibliographystyle{alpha} 

{\small \rmfamily{SCHOOL OF MATHEMATICS, UNIVERSITY OF BRISTOL, BRISTOL, UK, BS8 1TW}}

{\it E-mail address: }\texttt{corinna.ulcigrai@bristol.ac.uk}
\end{document}